\documentclass[11pt]{amsart}
\usepackage{graphicx}
\usepackage{amssymb}
\usepackage{amsmath}
\usepackage{amsthm}
\usepackage{mathtools}
\usepackage{tikz}
\usepackage{psfrag}
\usepackage{xcolor}

\newtheorem{thm}{Theorem}[section]
\newtheorem{lem}[thm]{Lemma}

\theoremstyle{definition}
\newtheorem{defn}[thm]{Definition}
\newtheorem{remk}[thm]{Remark}

\newcommand{\N}{\mathbb N}
\newcommand{\M}{\mathbb M}
\newcommand{\n}{\mathbf n}

\newcommand{\B}{\mathbf B}
\newcommand{\R}{\mathbb R}
\newcommand{\dist}{\operatorname{dist}}

\newcommand{\X}{\mathcal X}
\newcommand{\Y}{\mathcal Y}
\newcommand{\G}{\mathcal G}
\newcommand{\tr}{\operatorname{tr}}

\newcommand{\dv}{\operatorname{div}}

\newcommand{\wcon}{\rightharpoonup}

\makeatletter
\numberwithin{equation}{section}
\makeatother

\begin{document}

% author information

       % first author

       \author{Seonghak Kim}
       \address{Institute for Mathematical Sciences\\ Renmin University of China \\  Beijing 100872, PRC}
       \email{kimseo14@ruc.edu.cn}

       % second author

       \author{Baisheng Yan}

       % the address where the research was carried out
       \address{Department of Mathematics\\ Michigan State University\\ East Lansing, MI 48824, USA}

       % current address, usually not needed because it is the same as the
       % regular address

       \email{yan@math.msu.edu}

       % title

       \title[forward-backward parabolic equations]{On Lipschitz solutions for some  forward-backward parabolic  equations. II:\\ The case against Fourier}


\begin{abstract}
As a sequel to the paper \cite{KY2}, we study the existence and properties of Lipschitz solutions to the initial-boundary value problem  of some forward-backward parabolic equations  with diffusion fluxes violating Fourier's inequality.
\end{abstract}
\maketitle

\section{Introduction}
In this paper, following the authors' recent result \cite{KY2}, we further investigate the initial-boundary value problem
\begin{equation}\label{ib-P}
\begin{cases} u_t =\dv (A(Du) )& \mbox{in $\Omega_T$,} \\
A(Du)\cdot \n =0 & \mbox{on $\partial \Omega\times (0,T)$,}\\
u =u_0 & \mbox{on $\Omega\times \{t=0\},$}
\end{cases}
\end{equation}
where $\Omega\subset\R^n\;(n\ge 1)$ is a bounded domain, $T>0$ is any fixed number, $\Omega_T=\Omega\times (0,T)$, $\n$ is the outer unit normal on $\partial\Omega$, $u_0=u_0(x)$ is a given initial datum, and
$A=A(p):\R^n\to\R^n$ is the \emph{diffusion flux} (i.e., the negative of the heat flux density) of the evolution process. Here, $u=u(x,t)$ is the density of some quantity at position  $x$ and time $t$, with  $Du=(u_{x_1},\cdots,u_{x_n})$ and $u_t$ denoting its spatial gradient and    rate of change, respectively.

In general,  a  function  $u\in W^{1,\infty}(\Omega_T)$ is called a \emph{Lipschitz solution} to problem   (\ref{ib-P})   provided that equality
\begin{equation*}%\label{def:sol}
 \int_\Omega  (u(x,s)\zeta(x,s)-u_0(x) \zeta(x,0) )dx =
 \int_0^s \int_\Omega (u \zeta_t -A(Du)\cdot D\zeta ) dxdt
\end{equation*}
holds for  each $\zeta\in C^\infty(\bar\Omega_T)$ and each $s\in [0,T].$ Let $\zeta\equiv 1$; then it is immediate  from the definition that any Lipschitz solution $u$ to (\ref{ib-P}) conserves the total quantity over time:
\[
\int_\Omega  u(x,t)dx = \int_\Omega  u_0(x)dx \quad\forall t\in[0,T].
\]

The usual evolution of heat equation corresponds to the case of \emph{isotropic} diffusion given by Fourier's law: $A(p)=kp \; (p\in\R^n)$, where $k>0$ is the thermal diffusivity.
More generally, for standard diffusions, the diffusion flux $A(p)$ is assumed to be {\em monotone}; namely,
\begin{equation}\label{mono-0}
(A(p)-A(q))\cdot (p-q)\ge 0 \quad (p,\, q\in\R^n).
\end{equation}
In this case, problem (\ref{ib-P}) is parabolic and can be studied by the standard methods of parabolic equations, monotone operators and non-linear semigroup theory \cite{LSU, Br, Ln}. In particular,  when  $A(p)$ is given by a smooth convex  function $W(p)$ through
$A(p)=D_p W(p) \;  (p\in\R^n),$
the diffusion equation in (\ref{ib-P}) can be viewed and thus studied as a certain gradient flow generated by the energy functional
\[
I(u)=\int_\Omega W(Du(x))\,dx.
\]

On the other hand, for certain  applications of the evolution process  to some important physical problems, the underlying diffusion flux $A(p)$ may  be \emph{non-monotone}, yielding   non-parabolic problem (\ref{ib-P}).
In this regard, our recent paper \cite{KY2} studied the existence and properties of Lipschitz solutions to  (\ref{ib-P}) for some   non-monotone diffusion fluxes $A(p)$ of the form
\begin{equation}\label{fun-A}
A(p)=f(|p|^2)p\quad (p\in\R^n),
\end{equation}
given by a function $f\colon [0,\infty)\to \R$ with \emph{profile} $\sigma(s)=s f(s^2)$ having  one of the graphs  in Figures \ref{fig1} and  \ref{fig2}, referred to as the \emph{Perona-Malik type}  in image processing and the \emph{H\"ollig type} related to the phase transitions in thermodynamics, respectively; for more details, see \cite{KY2} and the references therein.

\begin{figure}[ht]
\begin{center}
\begin{tikzpicture}[scale = 1]
    \draw[->] (-.5,0) -- (11.3,0);
	\draw[->] (0,-.5) -- (0,5.2);
 \draw[dashed] (0,4.1)--(5,4.1);
    \draw[dashed] (5, 0)  --  (5, 4.1) ;
	\draw[thick]   (0, 0) .. controls  (2, 3.9) and (4.1,4) ..(5,4.1);
	\draw[thick]   (5, 4.1) .. controls  (7.2, 4) and (9,1) ..(11.1,0.5 );
%\draw[thick]   (0.9,1.5) .. controls  (1,1.6) and (2,2.6) ..(11,2.8);
	\draw (11.3,0) node[below] {$s$};
    \draw (-0.3,0) node[below] {{$0$}};
    \draw (10, 1.2) node[above] {$\sigma(s)$};
    \draw (0, 4.1) node[left] {$\sigma(s_0)$};
% \draw[dashed] (0,1.5)--(9.2, 1.5);
% \draw[dashed] (0,3)--(7.4, 3);
% \draw[dashed] (7.4,0)--(7.4, 3);
%\draw[dashed] (9.2,0)--(9.2, 1.5);
   \draw (5, 0) node[below] {$s_0$};
% \draw[dashed] (0.9,0)--(0.9, 1.5);
    \end{tikzpicture}
\end{center}
\caption{Perona-Malik type profiles   $\sigma(s)$.}
\label{fig1}
\end{figure}
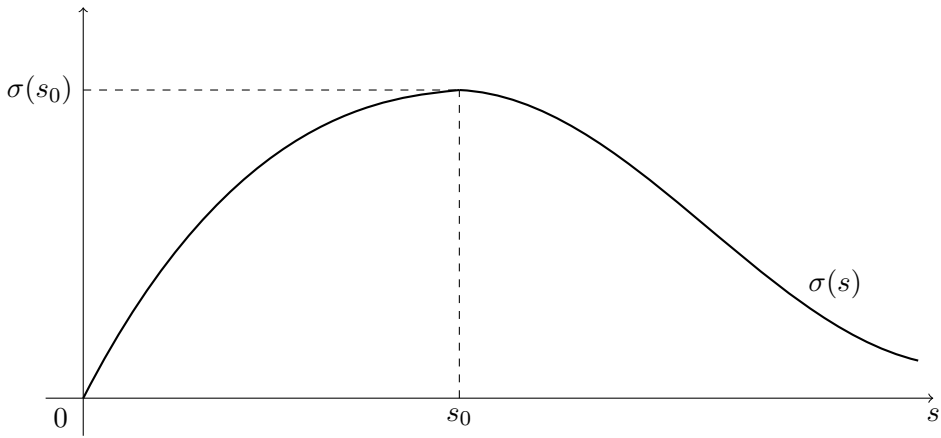

\begin{figure}[ht]
\begin{center}
\begin{tikzpicture}[scale =1]
    \draw[->] (-.5,0) -- (11.3,0);
	\draw[->] (0,-.5) -- (0,5.2);
 \draw[dashed] (0,1.7)--(5.6,1.7);
 \draw[dashed] (0,3.3)--(9.1,3.3);
 \draw[dashed] (9,0)--(9,3.3);
    \draw[dashed] (5.6, 0)  --  (5.6, 1.7) ;
   \draw[dashed] (0.85, 0)  --  (0.85, 1.7) ;
   \draw[dashed] (2.8, 0)  --  (2.8, 3.3) ;
	\draw[thick]   (0, 0) .. controls (2,5) and  (4, 3)   ..(5,2);
	\draw[thick]   (5, 2) .. controls  (6, 1) and (9,3) ..(11.1, 5 );
%\draw[thick]   (0.9,1.5) .. controls  (1,1.6) and (2,2.6) ..(11,2.8);
	\draw (11.3,0) node[below] {$s$};
    \draw (-0.3,0) node[below] {{$0$}};
    \draw (10.5, 3.5) node[above] {$\sigma(s)$};
 \draw (0, 3.3) node[left] {$\sigma(s_1)$};
 \draw (0, 1.7) node[left] {$\sigma(s_2)$};
% \draw[dashed] (0,1.5)--(9.2, 1.5);
% \draw[dashed] (0,3)--(7.4, 3);
% \draw[dashed] (7.4,0)--(7.4, 3);
%\draw[dashed] (9.2,0)--(9.2, 1.5);
   \draw (5.6, 0) node[below] {$s_2$};
   \draw (2.8, 0) node[below] {$s_1$};
   \draw (0.85, 0) node[below] {$s_1^*$};
   \draw (9, 0) node[below] {$s_2^*$};
% \draw[dashed] (0.9,0)--(0.9, 1.5);
    \end{tikzpicture}
\end{center}
\caption{H\"ollig  type  profiles  $\sigma(s).$}
\label{fig2}
\end{figure}
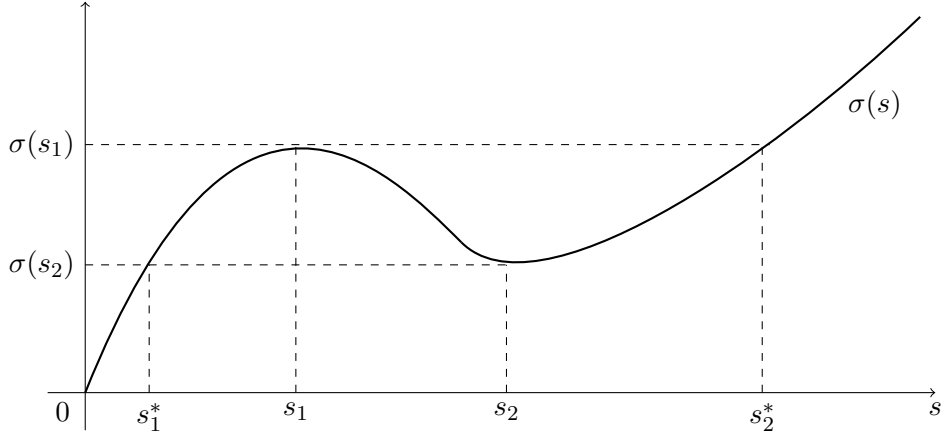

Parabolic and non-parabolic problems (\ref{ib-P}) discussed so far have used  diffusion fluxes $A(p)$ fulfilling \emph{Fourier's inequality}:
\begin{equation}\label{pos}
A(p)\cdot p\ge 0\quad (p\in\R^n),
\end{equation}
which is consistent to the Clausius-Duhem inequality in the second law of thermodynamics; see, e.g. \cite{Dy, Tr}. Observe that the inequality of monotonicity (\ref{mono-0}) implies Fourier's inequality (\ref{pos}), but the converse does not hold.

In this paper, we consider diffusion fluxes $A(p)$ of the form (\ref{fun-A}) with profiles $\sigma(s)$  having the graphs as in Figure \ref{fig3}; in this case, Fourier's inequality (\ref{pos}) is violated as
\[\left\{
\begin{array}{ll}
  A(p)\cdot p>0,  & |p|>s_0, \\
  A(p)\cdot p<0,  & 0<|p|<s_0, \\
  A(p)\cdot p=0,  & |p|\in\{0,\,s_0\}.
\end{array}\right.
\]
We will call such profiles $\sigma(s)$ as  \emph{non-Fourier type}.
More precisely, we impose the following conditions on the non-Fourier type profiles $\sigma(s)=sf(s^2)$.\\

\textbf{Hypothesis (NF):} (See Figure \ref{fig3}.)
\begin{itemize}
\item[(i)]  There exist two numbers $s_0>s_->0$ such that
\[
f\in  C^0([0,\infty))\cap C^{1}((0,s^2_-)\cup(s^2_-,s^2_0))\cap C^{1+\alpha}(s^2_0,\infty),
\]
where $0<\alpha<1$ is any fixed number.
\item[(ii)]
$\sigma(s_0)=0,\;\;\sigma'(s)<0\;\;\forall s\in(0,s_-),\; \sigma'(s)>0\;\;\forall s\in(s_-,s_0)\cup(s_0,\infty)$, and $\lambda\le \sigma'(s)\le \Lambda \;\;\forall s>2s_0$, where $\Lambda\ge\lambda>0$ are constants.
\item[(iii)] Let $s_+\in(s_0,\infty)$ denote the unique number with $\sigma(s_+)=-\sigma(s_-)$.
\end{itemize}
Also, for  each $r\in(0,\sigma(s_+))$, let $s_+(r)\in(s_0,s_+)$, $s_-^1(r)\in(0,s_-)$ and $s_-^2(r)\in(s_-,s_0)$ denote the unique numbers such that
\[
r=\sigma(s_+(r))=-\sigma(s_-^1(r))=-\sigma(s_-^2(r)).
\]

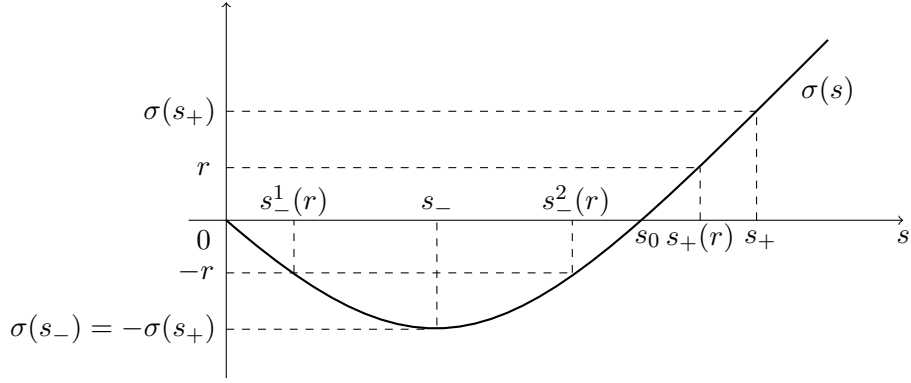
\begin{figure}[ht]
\begin{center}
\begin{tikzpicture}[scale =1]
    %\draw[->] (-.5,0) -- (11.3,0);
	\draw[->] (0,0.5) -- (0,5.5);
    \draw[->] (-.5,2.6) -- (9,2.6);
 \draw[dashed] (0,1.15)--(2.7,1.15);
 \draw[dashed] (0,4.05)--(7.05,4.05);
  \draw[dashed] (0,3.3)--(6.3,3.3);
  \draw[dashed] (0,1.9)--(4.65,1.9);
    \draw[dashed] (7.05, 2.6)  --  (7.05, 4.05) ;
    \draw[dashed] (6.3, 2.6)  --  (6.3, 3.3) ;
   \draw[dashed] (2.8, 1.2)  --  (2.8, 2.6) ;
   \draw[dashed] (0.9, 1.9)  --  (0.9, 2.6) ;
   \draw[dashed] (4.6, 1.9)  --  (4.6, 2.6) ;
	\draw[thick]   (0, 2.6) .. controls (3,0) and  (4, 1)   ..(8,5);
	%\draw[thick]   (5, 1) .. controls  (6, 1) and (9,3) ..(11.1, 5 );
%\draw[thick]   (0.9,1.5) .. controls  (1,1.6) and (2,2.6) ..(11,2.8);
	\draw (9,2.6) node[below] {$s$};
    \draw (-0.3,2.6) node[below] {{$0$}};
    \draw (-0.3,3.5) node[below] {$r$};
    \draw (-0.4,2.18) node[below] {$-r$};
    \draw (8, 4) node[above] {$\sigma(s)$};
 \draw (0, 4.05) node[left] {$\sigma(s_+)$};
 \draw (0, 1.15) node[left] {$\sigma(s_-)=-\sigma(s_+)$};
% \draw[dashed] (0,1.5)--(9.2, 1.5);
% \draw[dashed] (0,3)--(7.4, 3);
% \draw[dashed] (7.4,0)--(7.4, 3);
%\draw[dashed] (9.2,0)--(9.2, 1.5);
   \draw (5.6, 2.6) node[below] {$s_0$};
   \draw (0.9, 3.2) node[below] {$s_-^1(r)$};
   \draw (4.65, 3.2) node[below] {$s_-^2(r)$};
   \draw (2.8, 3.05) node[below] {$s_-$};
   \draw (6.3, 2.7) node[below] {$s_+(r)$};
   \draw (7.1, 2.6) node[below] {$s_+$};
% \draw[dashed] (0.9,0)--(0.9, 1.5);
    \end{tikzpicture}
\end{center}
\caption{Non-Fourier type profiles $\sigma(s)$.}
\label{fig3}
\end{figure}

The main purpose of this paper is to explore the scope of the methods of \cite{KY1, KY2} in the application to   problem (\ref{ib-P})   in all dimensions for  diffusion profiles $\sigma(s)$ of the non-Fourier type. To state our main theorem, we  make the following assumptions on the domain $\Omega$ and initial datum $u_0$:
\begin{equation}\label{assume-1}
\begin{cases} \mbox{$\Omega\subset\R^n$ is a bounded    domain  with $\partial \Omega$ of class $C^{2+\alpha}$,}\\
\mbox{$u_0\in C^{2+\alpha}(\bar\Omega)$ is non-constant with $Du_0\cdot \n |_{\partial \Omega}=0.$} \end{cases}
\end{equation}
In addition to this, we further assume without loss of generality that the initial datum $u_0$ satisfies
\begin{equation}\label{av-0}
\int_\Omega u_0(x)\,dx=0,
\end{equation}
since otherwise we may solve problem (\ref{ib-P}) with initial datum $\tilde u_0=u_0-\bar u_0,$ where $\bar u_0=\frac{1}{|\Omega|}\int_\Omega u_0\,dx$.

We now state the main existence theorem as follows.

\begin{thm}\label{thm:NF-1} Let $m_0=\min_{\bar \Omega}|Du_0|$ and $m_0'=\max\{m_0,s_0\}$, and assume $|Du_0(x_0)|\in (0,s_+)$ at some $x_0\in\Omega$. Then for each $\tilde r\in(\sigma(m_0'),\sigma(s_+))$, there exist an open set $\,\Omega_T^{\tilde r}\subset\Omega_T$  and infinitely many Lipschitz solutions $u$ to $(\ref{ib-P})$ of the following two types:

\emph{(\textbf{Type I})} $|Du|\in[s_-^2(\tilde r),s_+(\tilde r)]\cup\{0\}$ a.e. in $\Omega_T\setminus\Omega_T^{\tilde r}$.

\emph{(\textbf{Type II})} $|Du|\in[0,s_-^1(\tilde r)]\cup[s_0,s_+(\tilde r)]$ a.e. in $\Omega_T\setminus\Omega_T^{\tilde r}$.

\emph{(\textbf{Common for both types})}
\[
u\in C^{2+\alpha,1+\alpha/2}(\bar{\Omega}^{\tilde r}_T),\quad u_t=\dv(A(Du))\;\; \mbox{pointwise in}\;\;\Omega_T^{\tilde r},
\]
\[
|Du(x,t)|>s_+(\tilde r)\;\;\forall(x,t)\in\Omega_T^{\tilde r}, \quad\mbox{and}\quad \Omega_0^{\tilde r}\subset\partial\Omega_T^{\tilde r},
\]
where
$
\Omega_0^{\tilde r}=\{(x,0) \,|\, x\in\Omega,\,|Du_0(x)|>s_+(\tilde r) \}.
$
\end{thm}

Lipschitz solutions $u$ of \textbf{Type I} will be also called as \emph{forward-forward type}   or simply   \textbf{FFT} solutions, and likewise for those $u$ of \textbf{Type II}  as \emph{backward-forward type} or simply \textbf{BFT} solutions. For clear distinction of these two types, we keep using boldface letters for \textbf{Type I}, \textbf{Type II}, \textbf{FFT} and \textbf{BFT} throughout the paper.

As a byproduct of Theorem \ref{thm:NF-1}, we have the simple existence theorem as follows.

\begin{thm}\label{thm:NF-2} Let $\Omega$ be as in \emph{(\ref{assume-1})}.
Then for any initial datum $u_0\in C^{2+\alpha}(\bar\Omega)$ with $Du_0\cdot\n|_{\partial\Omega}=0$, problem  \emph{(\ref{ib-P})} has at least one Lipschitz solution.
\end{thm}

The rest of the paper is organized as follows. Section \ref{sec:approach} begins with a general density approach to problem (\ref{ib-P}) as a non-homogeneous partial differential inclusion. Then the general existence theorem, Theorem \ref{thm:main},   is formulated under a key density hypothesis. Also, some essential ingredients for the proof of the main theorem, Theorem \ref{thm:NF-1}, are  provided. As the pivotal analysis of the paper, the geometry of related matrix sets is investigated in Section \ref{sec:geometry}, leading to the relaxation result on a homogeneous differential inclusion, Theorem \ref{main-lemma}. Section \ref{sec:add-set} is devoted to the simultaneous construction of suitable boundary functions and admissible sets for \textbf{Types I} and \textbf{II} in the stream of the proof of Theorem \ref{thm:NF-1}. Then the key density hypothesis for each type is realized in Section \ref{sec:den-proof}, completing the proof of Theorem \ref{thm:NF-1}. The proof of Theorem \ref{thm:NF-2} is also included in the last part of this section. Lastly, Section \ref{sec:remarks} adds a remark on further existence results that can be deduced from the combination of \cite{KY2} and this paper.

%Throughout the paper, we use the boldface letters for \textbf{Cases I} and \textbf{II} for clear distinction. The paper follows a parallel exposition to deal with  both cases simultaneously. But for a better  readability, we strongly recommend for a reader to follow each case  one at a time.

\section{A general density approach and some useful results }\label{sec:approach}
In this section, we present a general density method and some essential ingredients for the proof of the main result, Theorem \ref{thm:NF-1}. All the proofs and motivational ideas can be found in the previous paper \cite{KY2} and references therein; so we do not repeat those here unless otherwise stated.

\subsection{Admissible set and the density approach}\label{subsec-approach}

For the general density approach to problem (\ref{ib-P}), we assume the following:
\[\left\{
\begin{array}{l}
  \mbox{$\Omega\subset\R^n$ is a bounded Lipschitz domain,} \\
  \mbox{the initial datum $u_0\in W^{1,\infty}(\Omega)$,} \\
  \mbox{the diffusion flux $A\in C(\R^n;\R^n)$.}
\end{array}\right.
\]

Assume that we have a function $\Phi=(u^*,v^*)\in W^{1,\infty}(\Omega_T;\R^{1+n})$  satisfying
\begin{equation}\label{bdry-1}
\begin{cases}
 u^*(x,0)=u_0(x), &x\in \Omega,\\
\dv v^*(x,t)=u^*(x,t), &  \mbox{a.e. $(x,t)\in\Omega_T$},
\\
v^*(\cdot,t) \cdot \n|_{\partial\Omega} =0, &  t\in [0,T],
\end{cases}
\end{equation}
which will be called a \emph{boundary function} for the initial datum $u_0.$
We denote by $W^{1,\infty}_{u^*}(\Omega_T),\, W^{1,\infty}_{v^*}(\Omega_T;\R^n)$ the usual \emph{Dirichlet classes} with boundary traces $u^*, \, v^*,$ respectively.

We say that $\mathcal U\subset W^{1,\infty}_{u^*}(\Omega_T)$ is  an \emph{admissible set}  provided that it is  nonempty and  bounded in $W^{1,\infty}_{u^*}(\Omega_T)$ and that for each $u\in \mathcal U$,  there exists a vector function  $v\in W_{v^*}^{1,\infty}(\Omega_T; \R^n)$ satisfying
\[
\mbox{$\dv v=u$ \, a.e. in $\Omega_T$\quad and\quad$\|v_t\|_{L^\infty(\Omega_T)}\le R$,}
\]
 where $R>0$ is any fixed number. If\; $\mathcal U$ is an admissible set,  for each  $\epsilon>0$,  let $\mathcal U_\epsilon$ be   the set of all  $u\in\mathcal U$ such that  there exists a  function $v\in W_{v^*}^{1,\infty}(\Omega_T; \R^n)$ satisfying
\[
\begin{split}
&\mbox{$\dv v=u$  \, a.e. in $\Omega_T$,\quad$\|v_t\|_{L^\infty(\Omega_T)}\le R$,}\quad \mbox{and}\\
&\quad\int_{\Omega_T} |v_t(x,t)-A(Du(x,t))|\,dxdt \leq\epsilon|\Omega_T|.\end{split}
\]

We now have the following general existence theorem under a pivotal density hypothesis   of\;  $\mathcal U_\epsilon$ in $\mathcal U.$ Although the proof of this theorem already appeared in \cite{KY1, KY2}, we need to reproduce it here since  the proof itself will be used in the proof of Theorem \ref{thm:NF-1}.

\begin{thm}\label{thm:main} Let\, $\mathcal U\subset W^{1,\infty}_{u^*}(\Omega_T)$ be an admissible set satisfying the  {\em density property:}
\begin{equation*}%\label{den-0}
\mbox{$\mathcal U_\epsilon$  is dense in $\mathcal U$ under the $L^\infty$-norm for each  $\epsilon>0$.}
\end{equation*}
Then,  given any $\varphi\in \mathcal U$, for each $\delta>0,$
there exists a Lipschitz solution $u\in  W_{u^*}^{1,\infty}(\Omega_T)$ to   $(\ref{ib-P})$ satisfying
$\|u-\varphi\|_{L^\infty(\Omega_T)}<\delta.$
Furthermore,  if $\mathcal U$ contains a function  which is not a Lipschitz solution to $(\ref{ib-P}),$ then $ (\ref{ib-P})$ itself admits  infinitely  many Lipschitz solutions.
\end{thm}

\begin{proof}
For clarity, we divide the proof into several steps.

1. Let $\X$ be the closure of\, $\mathcal U$ in the metric space $L^\infty(\Omega_T).$
Then $(\mathcal X,L^\infty)$ is a non-empty complete metric space.  By assumption, each $\mathcal U_\epsilon$ is dense in $\X.$ Moreover, since $\mathcal U$ is bounded in $W_{u^*}^{1,\infty}(\Omega_T)$,  we have  $\X\subset W_{u^*}^{1,\infty}(\Omega_T)$.

2. Let  $\Y =L^1(\Omega_T;\R^{n})$. For $h>0$, define $T_h\colon  \X \to \Y$ as follows. Given any $u\in \X$, write $u=u^* +w$ with $w\in W_0^{1,\infty}(\Omega_T)$ and define
\[
T_h (u) =Du^* + D(\rho_h * w),
\]
where $\rho_h(z)=h^{-N}\rho(z/h)$, with $z=(x,t)$ and $N=n+1$, is the standard $h$-mollifier in $\R^{N}$, and $\rho_h * w$ is the usual convolution in $\R^{N}$ with $w$ extended to be zero outside $\bar{\Omega}_T.$
Then, for each $h>0$, the map $T_h \colon   (\X, L^\infty) \to (\Y, L^1)$ is continuous, and for each $u\in \X$,
\[
\lim_{h\to 0^+} \|T_h (u)-Du\|_{L^1(\Omega_T)}=\lim_{h\to 0^+} \|\rho_h * Dw-Dw\|_{L^1(\Omega_T)}=0.
\]
Therefore, the spatial gradient operator
$D\colon \mathcal X\to \mathcal  Y$ is the pointwise limit of a sequence of continuous maps $T_h \colon \X\to \Y$; hence $D\colon \mathcal X\to \mathcal  Y$ is  a {\em Baire-one map}.
By Baire's category theorem (e.g., \cite[Theorem 10.13]{BBT}), there exists a
{\em residual set} $\G\subset \mathcal X$ such that the operator $D$ is
continuous at each point of $\mathcal G.$  Since $\X\setminus \G$  is of the {\em first category}, the set $\G$ is {\em dense} in $\X$. Therefore, given any $\varphi\in \X,$ for each $\delta>0$, there exists a function $u\in \G$ such that $\|u-\varphi\|_{L^\infty(\Omega_T)}<\delta.$

3. We now prove that each   $u\in \G$ is a Lipschitz solution to (\ref{ib-P}).   Let $u\in \G$ be given. By the density of $\mathcal U_\epsilon$ in $(\X,L^\infty)$ for each $\epsilon>0$, for every $j\in\mathbb N$, there exists a function $ u_j\in\mathcal U_{1/j}$ such that $\|u_j-u\|_{L^\infty(\Omega_T)} <1/j$. Since the  operator $D\colon (\X, L^\infty)\to (\Y, L^1)$ is continuous at $u$,  we have  $Du_j\to Du$ in $L^1(\Omega_T;\R^n).$  Furthermore, from (\ref{bdry-1}) and the definition of\, $\mathcal U_{1/j}$, there exists a  function $v_j\in W^{1,\infty}_{v^*}(\Omega_T;\R^n)$ such that for each $\zeta\in C^\infty(\bar\Omega_T)$ and each   $t\in [0,T],$
\begin{equation}\label{div-v3}
\begin{split} & \int_\Omega v_j(x,t)\cdot D\zeta(x,t)\,dx  =-\int_\Omega u_j(x,t)\zeta(x,t)\,dx,\\
 \|( & v_j)_t\|_{L^\infty(\Omega_T)}  \le R,\quad  \int_{\Omega_T} |(v_j)_t-A(Du_j)|\,dxdt \leq\frac{1}{j}|\Omega_T|.\end{split}
\end{equation}
Since $v_j(x,0)=v^*(x,0)\in  W^{1,\infty}(\Omega;\R^n)$ and  $\|(v_j)_t\|_{L^\infty(\Omega_T)}  \le R$, it follows that both sequences $\{v_j\}$ and  $\{(v_j)_t\}$ are bounded in $L^2(\Omega_T;\R^n)\approx L^2((0,T);L^2(\Omega;\R^n)).$  So we may  assume
\[
\mbox{$v_j \wcon v $ and $(v_j)_t\wcon v_t$ in $L^2((0,T);L^2(\Omega;\R^n))$}
\]
 for some $v\in W^{1,2}((0,T);L^2(\Omega;\R^n)),$ where $\wcon$ denotes the weak convergence.  Upon taking the limit as $j\to \infty$ in (\ref{div-v3}), since  $v\in C([0,T];L^2(\Omega;\R^n))$ and $A\in C(\R^n;\R^n)$,  we obtain
\[
\begin{split}
    \int_\Omega  v(x,t)&\cdot D\zeta(x,t)\,dx    =  -\int_\Omega u(x,t)\zeta(x,t)\,dx \quad (t\in [0,T]), \\
 &v_t(x,t)= A(Du(x,t)) \quad   a.e. \; (x,t)\in \Omega_T. \end{split}
\]
Consequently, by \cite[Lemma 3.1]{KY1}, $u$ is a Lipschitz solution to (\ref{ib-P}).

4. Finally, assume $\mathcal U$ contains a function which is not a Lipschitz solution to (\ref{ib-P}); hence   $\G\ne \mathcal U.$ Then $\G$ cannot be a finite set, since otherwise the $L^\infty$-closure $\X=\overline{\G}  =\overline{\mathcal U}$ would be a  finite set, making  $\mathcal U=\G.$ Therefore, in this case,   (\ref{ib-P}) admits infinitely many Lipschitz solutions.
The proof is complete.
\end{proof}

\subsection{Uniformly parabolic equations} We refer to the standard references (e.g., \cite{LSU, Ln}) for some notations concerning functions and domains of class $C^{k+\alpha}$ with an integer $k\ge 0$.

Assume $\tilde f\in C^{1+\alpha}([0,\infty))$ is a function satisfying
\begin{equation}\label{para}
  \theta \le \tilde f(s)+2s\tilde f'(s)\le \Theta \quad \forall\;  s\ge 0,
\end{equation}
where $\Theta\ge\theta>0$ are constants. This condition is equivalent to $\theta\le (s\tilde f(s^2))'\le \Theta$ for all $s\in\R;$ hence, $\theta\le \tilde f(s)\le\Theta$ for all $s\ge 0.$ Let
\[
\tilde A(p)=\tilde f(|p|^2)p \quad (p\in \R^n).
\]
Then we have
\[
\tilde A^i_{p_j}(p)  = \tilde f(|p|^2)\delta_{ij} + 2\tilde f'(|p|^2) p_ip_j \quad (i,j=1,2,\cdots,n;\; p\in \R^n)
\]
and hence  the \emph{uniform ellipticity condition}:
\begin{equation*}%\label{para-0}
\theta |q|^2 \le \sum_{i,j=1}^n \tilde A^i_{p_j}(p) q_iq_j\le \Theta |q|^2\quad \forall\; p,\;q\in\R^n.
\end{equation*}

The proof of the following classical result can be found in \cite[Theorem 13.24]{Ln}.

\begin{thm}\label{existence-gr-max} Assume  $(\ref{assume-1})$.
Then the initial-Neumann boundary value problem
\begin{equation}\label{ib-parabolic}
  \begin{cases}
  u_t=\dv (\tilde A(Du)) & \mbox{in }\Omega_T, \\
  {\partial u}/{\partial \n}=0 & \mbox{on }\partial\Omega\times(0,T), \\
  u(x,0)=u_0(x) & \mbox{for } x\in\Omega
\end{cases}
\end{equation}
admits a unique solution $u\in C^{2+\alpha,1+\alpha/2}(\bar\Omega_T)$.
\end{thm}

\subsection{Modified profile}\label{subsec-modi}

The following elementary result can be proved in a similar way as in \cite{Zh}; we omit the proof.

\begin{lem}[see Figure \ref{fig4}]\label{lem:modi-NF}
Assume \emph{Hypothesis (NF)}.
Then for each $0<r<\sigma(s_+)$, there exists a function  $\tilde\sigma\in C^{1+\alpha}([0,\infty))$ with $\tilde\sigma(0)=0$  such that $\tilde\sigma$ is linear near $0$ and that
\begin{equation*}%\label{modi-NF-1}
\begin{cases}
                     \tilde\sigma(s)>\sigma(s), & 0< s < s_+(r), \\
                     \tilde\sigma(s)=\sigma(s), & s_+(r)\le s <\infty, \\
                   \theta \le \tilde\sigma'(s) \le \Theta, & 0\leq s<\infty
\end{cases}
\end{equation*}
for some constants  $\Theta\ge\theta>0.$
With such a function $\tilde \sigma$, define $\tilde f(s)=\tilde\sigma(\sqrt s)/\sqrt s$\; $(s>0)$ and $\tilde f(0)=\lim_{s\to 0^+} \tilde\sigma(\sqrt s)/\sqrt s \,;$ then $\tilde f\in C^{1+\alpha}([0,\infty))$ fulfills condition $(\ref{para})$.
\end{lem}

\begin{figure}[ht]
\begin{center}
\begin{tikzpicture}[scale =1.2]
    %\draw[->] (-.5,0) -- (11.3,0);
	\draw[->] (0,0.5) -- (0,5.5);
    \draw[->] (-.5,2.6) -- (9,2.6);
 \draw[dashed] (0,1.15)--(2.7,1.15);
 \draw[dashed] (0,4.05)--(7.05,4.05);
  \draw[dashed] (0,3.3)--(6.3,3.3);
  \draw[dashed] (0,1.9)--(4.65,1.9);
    \draw[dashed] (7.05, 2.6)  --  (7.05, 4.05) ;
    \draw[dashed] (6.3, 2.6)  --  (6.3, 3.3) ;
   \draw[dashed] (2.8, 1.2)  --  (2.8, 2.6) ;
   \draw[dashed] (0.9, 1.9)  --  (0.9, 2.6) ;
   \draw[dashed] (4.6, 1.9)  --  (4.6, 2.6) ;
	\draw[thick]   (0, 2.6) .. controls (3,0) and  (4, 1)   ..(8,5);
   \draw[thick,red] (0,2.6) -- (2,2.75);
   \draw[thick,red] (2,2.75) .. controls (3,2.8) and  (5.5, 2.85)   ..(6.3,3.3);
	%\draw[thick]   (5, 1) .. controls  (6, 1) and (9,3) ..(11.1, 5 );
%\draw[thick]   (0.9,1.5) .. controls  (1,1.6) and (2,2.6) ..(11,2.8);
	\draw (9,2.6) node[below] {$s$};
    \draw (-0.3,2.6) node[below] {{$0$}};
    \draw (-0.3,3.5) node[below] {$r$};
    \draw (-0.4,2.18) node[below] {$-r$};
    \draw (3, 2.8) node[above] {\textcolor{red}{$\tilde\sigma(s)$}};
    \draw (8, 4) node[above] {$\sigma(s)$};
 \draw (0, 4.05) node[left] {$\sigma(s_+)$};
 \draw (0, 1.15) node[left] {$\sigma(s_-)$};
% \draw[dashed] (0,1.5)--(9.2, 1.5);
% \draw[dashed] (0,3)--(7.4, 3);
% \draw[dashed] (7.4,0)--(7.4, 3);
%\draw[dashed] (9.2,0)--(9.2, 1.5);
   \draw (5.6, 2.6) node[below] {$s_0$};
   %\draw (0.9, 3.2) node[below] {$s_-^1(r)$};
   %\draw (4.65, 3.2) node[below] {$s_-^2(r)$};
   %\draw (2.8, 3.05) node[below] {$s_-$};
   \draw (6.3, 2.7) node[below] {$s_+(r)$};
   \draw (7.1, 2.6) node[below] {$s_+$};
% \draw[dashed] (0.9,0)--(0.9, 1.5);
    \end{tikzpicture}
\end{center}
\caption{Non-Fourier type profile $\sigma(s)$ and modified  \textcolor{red}{$\tilde\sigma(s)$}.}
\label{fig4}
\end{figure}
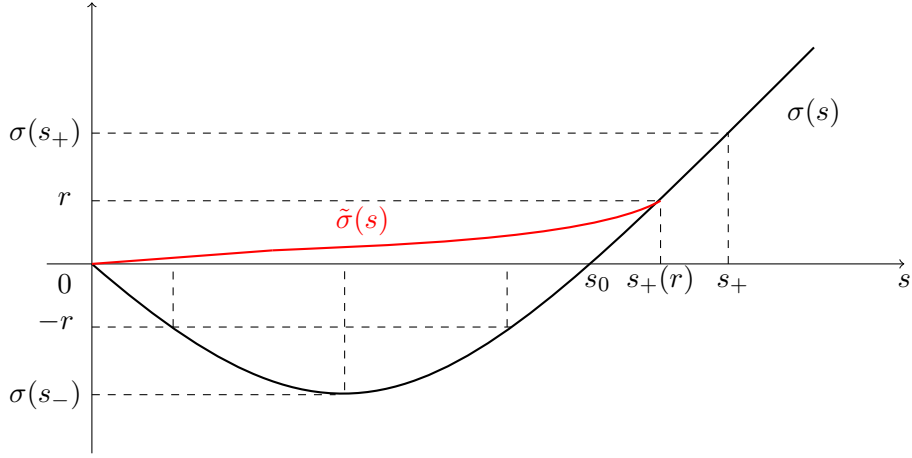

\subsection{Right inverse of the  divergence operator} We follow an argument of Bourgain and Brezis \cite[Lemma 4]{BB} to construct  a right inverse $\mathcal R$ of the divergence operator: $\dv \mathcal R=\mbox{Id}$ (in the sense of distributions in $\Omega_T$). For the purpose of this paper, the construction of $\mathcal R$ is restricted to a {\em box}, by which we mean  a domain $Q$ given by $Q=J_1\times J_2\times \cdots\times J_n$, where $J_i=(a_i,b_i)\subset\R$ is a finite open  interval.

We have the following result \cite[Theorem 2.3]{KY1}.

\begin{thm}\label{div-inv} Let $Q\times I$ be a box in $\R^{n+1}$, where $I\subset\R$ is a finite open interval. Then there is a bounded linear operator $\mathcal R=\mathcal R_n \colon  L^\infty(Q\times I)\to L^\infty(Q\times I;\R^n)$ satisfying the following: If $u\in W^{1,\infty}_0(Q\times I)$ is such that  $\int_{Q}u(x,t)\,dx=0$ for all $t\in I$, then $v:=\mathcal R u\in W^{1,\infty}_0(Q\times I;\R^n)$,  $\dv v=u$ a.e.\,in $Q\times I$, and
\begin{equation}\label{div-1}
\|v_t \|_{L^\infty(Q\times I)}  \le C_n \,(|J_1|+\cdots+|J_n|) \|u_t\|_{L^\infty(Q\times I)},
\end{equation}
where $Q=J_1\times\cdots\times J_n$ and $C_n>0$ is a dimensional constant.  Moreover, if $u\in C^1(\overline{Q\times I})$, then $v  \in C^1(\overline{Q\times I};\R^{n}).$
\end{thm}

\section{Geometry of the relevant matrix sets}\label{sec:geometry}

Let $A(p)$ be the diffusion flux given by (\ref{fun-A}) with profile $\sigma(s)=s f(s^2)$ satisfying Hypothesis (NF). Let $K_0$ be the subset of the $(1+n)\times(n+1)$ matrix space $\M^{(1+n)\times(n+1)}$ defined by
\begin{equation*}%\label{set-K}
 K_0= \left\{ \begin{pmatrix} p & c\\ B & A(p)\end{pmatrix}\,\Big | \,
  p\in\R^n,\, c\in\R,\,B\in\mathbb{M}^{n\times n},\,  \tr B=0
 \right\}.
\end{equation*}

Under Hypothesis  (NF), certain structures  of  the set $K_0$ turn out to be  quite useful, especially when it comes to the relaxation of homogeneous partial differential inclusion $\nabla \omega(z)\in K_0$ with $z=(x,t)$ and $\omega=(\varphi,\psi)$. We investigate these structures and establish such a relaxation result for both \textbf{FFT} and \textbf{BFT} solutions throughout this section.

\subsection{Geometry of the matrix set $K_0$} We study some subsets of $K_0$, depending on the different types of solutions to be sought.

\subsubsection*{\bf Type I: FFT solutions} In search for this type of solutions, we adopt the following notations. Fix any two numbers $0<r_1<r_2<\sigma(s_+)$, and let $F_0=F_{0,r_1,r_2}$ be the subset of $K_0$ defined by
\[
F_0=\left\{ \begin{pmatrix} p & c\\ B & A(p)\end{pmatrix}\,\Big|\, \begin{array}{l}
                                                                                         p\in\R^n, \, |p|\in(s^2_-(r_2),s^2_-(r_1))\cup(s_+(r_1),s_+(r_2)), \\
                                                                                         c\in\R, \, B\in \mathbb M^{n\times n},\, \tr B=0
                                                                                       \end{array}
   \right\}.
\]
We decompose the set $F_0$ into two disjoint subsets as follows:
\[
F_-=F_{-,r_1,r_2}=\left\{ \begin{pmatrix} p & c\\ B & A(p)\end{pmatrix}\,\Big|\, \begin{array}{l}
                                                                                         p\in\R^n, \, |p|\in(s^2_-(r_2),s^2_-(r_1)), \\
                                                                                         c\in\R, \, B\in \mathbb M^{n\times n},\, \tr B=0
                                                                                       \end{array}
   \right\},
\]
\[
F_+=F_{+,r_1,r_2}=\left\{ \begin{pmatrix} p & c\\ B & A(p)\end{pmatrix}\,\Big|\, \begin{array}{l}
                                                                                         p\in\R^n, \, |p|\in(s_+(r_1),s_+(r_2)), \\
                                                                                         c\in\R, \, B\in \mathbb M^{n\times n},\, \tr B=0
                                                                                       \end{array}
   \right\}.
\]

\subsubsection*{\bf Type II: BFT solutions} To handle this type of solutions, we use the following notations. Fix any two numbers $0<r_1<r_2<\sigma(s_+)$, and let $F_0=F_{0,r_1,r_2}$ be the subset of $K_0$ given by
\[
F_0=\left\{ \begin{pmatrix} p & c\\ B & A(p)\end{pmatrix}\,\Big|\, \begin{array}{l}
                                                                                         p\in\R^n, \, |p|\in(s^1_-(r_1),s^1_-(r_2))\cup(s_+(r_1),s_+(r_2)), \\
                                                                                         c\in\R, \, B\in \mathbb M^{n\times n},\, \tr B=0
                                                                                       \end{array}
   \right\}.
\]
The set $F_0$ is then decomposed into two disjoint subsets as follows:
\[
F_-=F_{-,r_1,r_2}=\left\{ \begin{pmatrix} p & c\\ B & A(p)\end{pmatrix}\,\Big|\, \begin{array}{l}
                                                                                         p\in\R^n, \, |p|\in(s^1_-(r_1),s^1_-(r_2)), \\
                                                                                         c\in\R, \, B\in \mathbb M^{n\times n},\, \tr B=0
                                                                                       \end{array}
   \right\},
\]
\[
F_+=F_{+,r_1,r_2}=\left\{ \begin{pmatrix} p & c\\ B & A(p)\end{pmatrix}\,\Big|\, \begin{array}{l}
                                                                                         p\in\R^n, \, |p|\in(s_+(r_1),s_+(r_2)), \\
                                                                                         c\in\R, \, B\in \mathbb M^{n\times n},\, \tr B=0
                                                                                       \end{array}
   \right\}.
\]

%Let
%\[
%\mathcal K=\left\{(p,\sigma(|p|)\frac{p}{|p|}) \,\Big|\, |p|\in[s_1^*,s_1]\cup[s_2,s_2^*]\right\}.
%\]

In order to study the homogeneous differential inclusion $\nabla \omega(z)\in K_0$, we first scrutinize the rank-one structure of the set $F_0\subset K_0$ for each type. To this aim, we define
\[
R(F_0)=\bigcup_{\xi_\pm\in F_\pm,\,\mathrm{rank}(\xi_+-\xi_-)=1}(\xi_-,\xi_+),
\]
where $(\xi_-,\xi_+)$ is the open line segment in $\mathbb{M}^{(1+n)\times(n+1)}$ joining $\xi_\pm$.
We now explore the structure of the set $R(F_0)$ in detail for both types simultaneously.

\subsubsection{\bf Alternate expression for $R(F_0)$} We  establish more specific criteria for  matrices in  $R(F_0)$ than its definition.
\begin{lem}\label{lem-form}
Let $\xi\in\mathbb{M}^{(1+n)\times(n+1)}$. Then
$\xi\in R(F_0)$ if and only if there exist numbers $t_-<0<t_+$ and vectors $q,\,\gamma\in\R^n$ with $|q|=1,\,\gamma\cdot q=0$ such that for each $b\in\R\setminus\{0\}$, if
$\eta=\begin{pmatrix} q & b\\ \frac{1}{b}q\otimes\gamma & \gamma\end{pmatrix}$, then
$\xi+t_\pm\eta\in F_\pm$.
\end{lem}

\begin{proof}
Assume $\xi=\begin{pmatrix} p & c\\ B & \beta\end{pmatrix}\in R(F_0)$. By definition, $
\xi+t_\pm\tilde\eta\in F_\pm,
$
where $t_-<0<t_+$ and $\tilde\eta$ is a rank-one matrix given by
\[
\tilde\eta=\begin{pmatrix} a\\ \alpha \end{pmatrix}\otimes (q,\tilde b)=\begin{pmatrix} aq & a\tilde b\\ \alpha\otimes q & \tilde b\alpha\end{pmatrix},\quad a^2+|\alpha|^2\neq0, \quad \tilde b^2+|q|^2\neq0,
\]
for some $a,\,\tilde b\in\R$ and $\alpha,\,q\in\R^n$; here $\alpha\otimes q$ denotes the rank-one or zero matrix $(\alpha_iq_j)$ in $\mathbb{M}^{n\times n}$. Condition $\xi+t_\pm\tilde\eta\in F_\pm$ with $t_-<0<t_+$ is equivalent to the following:
For \textbf{Type I},
\begin{equation}\label{form-PM}
\begin{split}
&\tr B=0,\quad \alpha\cdot q=0,\quad A(p+t_\pm aq)=\beta+t_\pm \tilde b\alpha,\\  |p+t&_+ aq|\in(s_+(r_1),s_+(r_2)),\quad |p+t_- aq|\in(s^2_-(r_2),s^2_-(r_1)).
\end{split}
\end{equation}
For \textbf{Type II},
\begin{equation}\label{form-C}
\begin{split}
&\tr B=0,\quad \alpha\cdot q=0,\quad A(p+t_\pm aq)=\beta+t_\pm \tilde b\alpha,\\  |p+t&_+ aq|\in(s_+(r_1),s_+(r_2)),\quad |p+t_- aq|\in(s^1_-(r_1),s^1_-(r_2)).
\end{split}
\end{equation}
Therefore, $aq\neq 0$. Upon rescaling $\tilde\eta$ and $t_\pm$, we can assume $a=1$ and $|q|=1$; namely,
\[
\tilde\eta=\begin{pmatrix} q & \tilde b\\ \alpha\otimes q & \tilde b\alpha\end{pmatrix},\quad |q|=1,\quad \alpha\cdot q =0.
\]
We now set $\gamma=\tilde b\alpha$. Let $b\in\R\setminus\{0\}$ and
\begin{equation*}
\eta=\begin{pmatrix} q & b\\ \frac{1}{b}\gamma\otimes q & \gamma\end{pmatrix}.
\end{equation*}
From (\ref{form-PM}) (\textbf{Type I}), (\ref{form-C}) (\textbf{Type II}), it follows that $\xi+t_\pm\eta\in F_\pm$.

The converse easily follows from the definition of $R(F_0)$.
\end{proof}

\subsubsection{\bf Diagonal components of matrices in $R(F_0)$.}
The following gives a  description for the diagonal components of matrices in  $R(F_0)$ for both types; the proof is precisely the same as that of \cite[Lemma 5.3]{KY2} and thus not reproduced here.
\begin{lem}\label{lem-rough}
\begin{equation}\label{rough-1}
R(F_0)=\left\{ \begin{pmatrix} p & c\\ B & \beta\end{pmatrix}\,\Big|\, c\in\R,\,B\in\mathbb{M}^{n\times n},\,\tr B=0, \, (p,\beta)\in\mathcal{S}
   \right\}
\end{equation}
for some set $\mathcal{S}=\mathcal{S}_{r_1,r_2}\subset\R^{n+n}$.
\end{lem}

\subsubsection{\bf Selection of approximate collinear rank-one connections for $R(F_0)$.}
We first give a 2-dimensional description for the rank-one connections of diagonal components of matrices in $R(F_0)$ in a general form. The following lemma is common for both \textbf{Types  I} and  \textbf{II}.

\begin{lem}\label{lem-2d-rank}
For all positive numbers $a,b$ and $c$ with $b>a$, there exists a nonnegative continuous function $h(a,b,c;\cdot,\cdot,\cdot,\cdot,\cdot,\cdot)$ defined on
\[
I_{a,b,c}:=[0,a)\times\left[0,\frac{b-a}{2}\right)\times\left[0,\frac{b-a}{2}\right)\times[0,\infty)\times[0,c)\times[0,\infty)
\]
with $h(a,b,c;0,0,0,0,0,0)=0$
satisfying the following:

Let $\delta_{11},\delta_{12},\delta_{21},\delta_{22},\eta_1$ and $\eta_2$ be any positive numbers with
\[
0<a-\delta_{11}<a<a+\delta_{12}<\frac{a+b}{2}<b-\delta_{21},\quad 0<c-\eta_1,
\]
and let $R_1\in[a-\delta_{11},a+\delta_{12}]$, $R_2\in[b-\delta_{21},b+\delta_{22}]$, and $\tilde R_1,\tilde R_2\in[c-\eta_1,c+\eta_2]$.
Suppose $\theta\in[-\pi/2,\pi/2]$ and
\[
\begin{split}
\Big(\tilde R_1\big( & \cos(\frac{\pi}{2}+\theta),\sin(\frac{\pi}{2}+\theta)\big)-\tilde R_2\big(\cos(\frac{\pi}{2}-\theta),\sin(\frac{\pi}{2}-\theta)\big)\Big)\\
 & \cdot\Big(R_1\big(\cos(\theta-\frac{\pi}{2}),\sin(\theta-\frac{\pi}{2})\big)-R_2\big(\cos(\frac{\pi}{2}-\theta),\sin(\frac{\pi}{2}-\theta)\big)\Big)=0.
\end{split}
\]
Then $-\frac{\pi}{2}<\theta<\frac{\pi}{2}$, $\tilde R_1\ge \tilde R_2$, and
\[
\max\Big\{\big|(0,-a)-R_1\big(\cos(\theta-\frac{\pi}{2}),\sin(\theta-\frac{\pi}{2})\big)\big|,\,\big|(0,b)-R_2\big(\cos(\frac{\pi}{2}-\theta),\sin(\frac{\pi}{2}-\theta)\big)\big|,
\]
\[
\big|(0,c)-\tilde R_1\big(\cos(\frac{\pi}{2}+\theta),\sin(\frac{\pi}{2}+\theta)\big)\big|,\,\big|(0,c)-\tilde R_2\big(\cos(\frac{\pi}{2}-\theta),\sin(\frac{\pi}{2}-\theta)\big)\big|\Big\}
\]
\[
\le h(a,b,c;\delta_{11},\delta_{12},\delta_{21},\delta_{22},\eta_1,\eta_2).
\]
\end{lem}
\begin{proof}
By assumption,
\[
0=(\tilde R_1(-\sin\theta,\cos\theta)-\tilde R_2(\sin\theta,\cos\theta))\cdot(R_1(\sin\theta,-\cos\theta)-R_2(\sin\theta,\cos\theta))
\]
\[
=(-(\tilde R_1+\tilde R_2)\sin\theta,(\tilde R_1-\tilde R_2)\cos\theta)\cdot((R_1-R_2)\sin\theta,-(R_1+R_2)\cos\theta)
\]
\[
=(\tilde R_1+\tilde R_2)(R_2-R_1)\sin^2\theta-(\tilde R_1-\tilde R_2)(R_1+R_2)\cos^2\theta,
\]
that is,
\[
(\tilde R_1-\tilde R_2)(R_1+R_2)\cos^2\theta=(\tilde R_1+\tilde R_2)(R_2-R_1)\sin^2\theta;
\]
hence, $\theta\ne \pm\frac{\pi}{2}$, $\tilde R_1\ge \tilde R_2$, and
\[
\theta=\pm\tan^{-1}\left(\sqrt\frac{(\tilde R_1-\tilde R_2)(R_1+R_2)}{(\tilde R_1+\tilde R_2)(R_2-R_1)}\right).
\]
So
\[
\begin{split}
|\theta| & \le\tan^{-1}\left(\sqrt{\frac{(a+b+\delta_{12}+\delta_{22})(\eta_1+\eta_2)}{2(b-a-\delta_{12}-\delta_{21})(c-\eta_1)}}\right) \\
& =:g(a,b,c;\delta_{11},\delta_{12},\delta_{21},\delta_{22},\eta_1,\eta_2).
\end{split}
\]
Note that the function $g(a,b,c;\cdot,\cdot,\cdot,\cdot,\cdot,\cdot):I_{a,b,c}\to[0,\pi/2)$ is well-defined and continuous and that $g(a,b,c;\delta_{11},\delta_{12},\delta_{21},\delta_{22},\eta_1,\eta_2)=0$\, whenever
\[(\delta_{11},\delta_{12},\delta_{21},\delta_{22},\eta_1,\eta_2)\in I_{a,b,c},\;\;\eta_1=\eta_2=0.
\]

Observe now that
\[
|(0,-a)-R_1(\cos(\theta-\frac{\pi}{2}),\sin(\theta-\frac{\pi}{2}))|
\]
\[
\le\max\{|(0,-a)-(a+\delta_{12})(\sin\theta,-\cos\theta)|,
|(0,-a)-(a-\delta_{11})(\sin\theta,-\cos\theta)|\}
\]
\[
\begin{split}
=\max\Big\{ & \sqrt{(a+\delta_{12})^2\sin^2\theta+(a-(a+\delta_{12})\cos\theta)^2},\\
& \quad\quad\quad\quad \sqrt{(a-\delta_{11})^2\sin^2\theta+(a-(a-\delta_{11})\cos\theta)^2}\Big\}
\end{split}
\]
\[
\begin{split}
=\max\Big\{ & \sqrt{(a+\delta_{12})^2+a^2-2a(a+\delta_{12})\cos\theta},\\
& \quad\quad\quad\quad \sqrt{(a-\delta_{11})^2+a^2-2a(a-\delta_{11})\cos\theta}\Big\}
\end{split}
\]
\[
\begin{split}
\le\max\Big\{ & \sqrt{(a+\delta_{12})^2+a^2-2a(a+\delta_{12})\cos(g(a,b,c;\delta_{11},\delta_{12},\delta_{21},\delta_{22},\eta_1,\eta_2))},\\
&  \sqrt{(a-\delta_{11})^2+a^2-2a(a-\delta_{11})\cos(g(a,b,c;\delta_{11},\delta_{12},\delta_{21},\delta_{22},\eta_1,\eta_2))}\Big\}
\end{split}
\]
\[
=:h_{1}(a,b,c;\delta_{11},\delta_{12},\delta_{21},\delta_{22},\eta_1,\eta_2),
\]
\[
|(0,b)-R_2(\cos(\frac{\pi}{2}-\theta),\sin(\frac{\pi}{2}-\theta))|
\]
\[
\le\max\{|(0,b)-(b+\delta_{22})(\sin\theta,\cos\theta)|,
|(0,b)-(b-\delta_{21})(\sin\theta,\cos\theta)|\}
\]
\[
\begin{split}
=\max\Big\{ & \sqrt{(b+\delta_{22})^2\sin^2\theta+(b-(b+\delta_{22})\cos\theta)^2},\\
& \quad\quad\quad\quad \sqrt{(b-\delta_{21})^2\sin^2\theta+(b-(b-\delta_{21})\cos\theta)^2}\Big\}
\end{split}
\]
\[
\begin{split}
=\max\Big\{ & \sqrt{(b+\delta_{22})^2+b^2-2b(b+\delta_{22})\cos\theta},\\
& \quad\quad\quad\quad \sqrt{(b-\delta_{21})^2+b^2-2b(b-\delta_{21})\cos\theta}\Big\}
\end{split}
\]
\[
\begin{split}
\le\max\Big\{ & \sqrt{(b+\delta_{22})^2+b^2-2b(b+\delta_{22})\cos(g(a,b,c;\delta_{11},\delta_{12},\delta_{21},\delta_{22},\eta_1,\eta_2))},\\
&  \sqrt{(b-\delta_{21})^2+b^2-2b(b-\delta_{21})\cos(g(a,b,c;\delta_{11},\delta_{12},\delta_{21},\delta_{22},\eta_1,\eta_2))}\Big\}
\end{split}
\]
\[
=:h_{2}(a,b,c;\delta_{11},\delta_{12},\delta_{21},\delta_{22},\eta_1,\eta_2),
\]
\[
|(0,c)-\tilde R_1(\cos(\frac{\pi}{2}+\theta),\sin(\frac{\pi}{2}+\theta))|
\]
\[
\begin{split}
\le\max\Big\{ & \sqrt{(c+\eta_2)^2+c^2-2c(c+\eta_2)\cos(g(a,b,c;\delta_{11},\delta_{12},\delta_{21},\delta_{22},\eta_1,\eta_2))},\\
&  \sqrt{(c-\eta_1)^2+c^2-2c(c-\eta_1)\cos(g(a,b,c;\delta_{11},\delta_{12},\delta_{21},\delta_{22},\eta_1,\eta_2))}\Big\}
\end{split}
\]
\[
=:h_{3}(a,b,c;\delta_{11},\delta_{12},\delta_{21},\delta_{22},\eta_1,\eta_2),
\]
\[
|(0,c)-\tilde R_2(\cos(\frac{\pi}{2}-\theta),\sin(\frac{\pi}{2}-\theta))|\le h_{3}(a,b,c;\delta_{11},\delta_{12},\delta_{21},\delta_{22},\eta_1,\eta_2).
\]
Define
\[
h(a,b,c;\delta_{11},\delta_{12},\delta_{21},\delta_{22},\eta_1,\eta_2)=\max_{1\le j\le 3} h_{j}(a,b,c;\delta_{11},\delta_{12},\delta_{21},\delta_{22},\eta_1,\eta_2).
\]
Then it is easy to see that the function $h(a,b,c;\cdot,\cdot,\cdot,\cdot,\cdot,\cdot):I_{a,b,c}\to[0,\infty)$ is well-defined and satisfies the required properties.
\end{proof}

We now apply the previous lemma to  choose \emph{approximate} collinear rank-one connections for the diagonal components of matrices in $R(F_0)$.

\begin{thm}\label{lem-inv}
For each $0<r<\sigma(s_+)$, there exists a number $\mu_r>0$ with $0<r-\mu_r<r+\mu_r<\sigma(s_+)$ satisfying the following:

Let $0<\mu\le\mu_r$, and let $p_\pm\in\R^n$ satisfy
\[
s^2_-(r+\mu)<|p_-|<s^2_-(r-\mu)<s_+(r-\mu)<|p_+|<s_+(r+\mu),\;\;\mbox{\emph{(\textbf{Type I})}}
\]
\[
s^1_-(r-\mu)<|p_-|<s^1_-(r+\mu)<s_+(r-\mu)<|p_+|<s_+(r+\mu)\;\;\mbox{\emph{(\textbf{Type II})}}
\]
and
$(A(p_+)-A(p_-))\cdot(p_+-p_-)=0$. Then there exists a vector $\zeta^0\in\mathbb{S}^{n-1}$ such that, with $p^0_{+}=s_+(r)\zeta^0$, $p^0_{-}=-s^2_-(r)\zeta^0$ \emph{(\textbf{Type I})}, $p^0_{-}=-s^1_-(r)\zeta^0$ \emph{(\textbf{Type II})},
$A(p^0_\pm)=r\zeta^0$, we have
\[
\begin{split}
\max & \{|p^0_--p_-|,|p^0_+-p_+|,|A(p^0_-)-A(p_-)|,|A(p^0_+)-A(p_+)|\} \\
& \le h\big(s^2_-(r),s_+(r),r; s^2_-(r)-s^2_-(r+\mu),s^2_-(r-\mu)-s^2_-(r), \\
& \quad\quad\; s_+(r)-s_+(r-\mu),s_+(r+\mu)-s_+(r),\mu,\mu\big),\;\;\emph{(\textbf{Type I})}
\end{split}
\]
\[
\begin{split}
\max & \{|p^0_--p_-|,|p^0_+-p_+|,|A(p^0_-)-A(p_-)|,|A(p^0_+)-A(p_+)|\} \\
& \le h\big(s^1_-(r),s_+(r),r; s^1_-(r)-s^1_-(r-\mu),s^1_-(r+\mu)-s^1_-(r), \\
& \quad\quad\; s_+(r)-s_+(r-\mu),s_+(r+\mu)-s_+(r),\mu,\mu\big),\;\;\emph{(\textbf{Type II})}
\end{split}
\]
where $h$ is the function in Lemma \ref{lem-2d-rank}.
\end{thm}

\begin{proof}
Fix any $0<r<\sigma(s_+)$. Since
\[
\lim_{\mu\to 0^+}s^2_-(r-\mu)=s^2_-(r)<\frac{s^2_-(r)+s_+(r)}{2}<s_+(r)=\lim_{\mu\to 0^+}s_+(r-\mu),\;(\textbf{Type I})
\]
\[
\lim_{\mu\to 0^+}s^1_-(r+\mu)=s^1_-(r)<\frac{s^1_-(r)+s_+(r)}{2}<s_+(r)=\lim_{\mu\to 0^+}s_+(r-\mu),\;(\textbf{Type II})
\]
we can find a $\mu_r>0$ so small that $0<r-\mu_r<r+\mu_r<\sigma(s_+)$ and that for every $0<\mu\le\mu_r$, we have
\[
s^2_-(r-\mu) <\frac{s^2_-(r)+s_+(r)}{2}< s_+(r-\mu),\;\;(\textbf{Type I})
\]
\[
s^1_-(r+\mu) <\frac{s^1_-(r)+s_+(r)}{2}< s_+(r-\mu).\;\;(\textbf{Type II})
\]

Now, let $0<\mu\le\mu_r$, and let $p_\pm\in\R^n$ satisfy the conditions in the statement of the theorem.
Let $\Sigma_2$ denote the 2-dimensional linear subspace of $\R^n$ spanned by the two vectors $p_\pm$. (In  case of collinear  $p_\pm$, we choose $\Sigma_2$ to be any 2-dimensional linear space in $\R^n$ containing $p_\pm$.) From the orthogonality condition, we have
\[
\sigma(|p_+|)|p_+|+\sigma(|p_-|)|p_-|-\Big(\frac{\sigma(|p_+|)}{|p_+|}+\frac{\sigma(|p_-|)}{|p_-|}\Big)(p_+\cdot p_-)=0.
\]
If $p_\pm$ were pointing in the same direction, we would have $\sigma(|p_+|)=\sigma(|p_-|)$, a contradiction. Thus we can set
\[
\zeta^0=\frac{\frac{p_+}{|p_+|}-\frac{p_-}{|p_-|}}{\big|\frac{p_+}{|p_+|}-\frac{p_-}{|p_-|}\big|}\in\mathbb{S}^{n-1}\cap\Sigma_2.
\]
Since the vectors $p_\pm$, $A(p_\pm)$ and $\zeta^0$ all lie in $\Sigma_2$, we can recast the problem into the setting of the previous lemma via one of the two  linear isomorphisms of
$\Sigma_2$ onto $\R^2$ with correspondence $\zeta^0\leftrightarrow (0,1)\in\R^2.$ Then the result follows with the choices below in applying Lemma \ref{lem-2d-rank}: $a=s^2_-(r)$ (\textbf{Type I}), $a=s^1_-(r)$ (\textbf{Type II}), $b=s_+(r)$, $c=r$, $\delta_{11}=s^2_-(r)-s^2_-(r+\mu)$ (\textbf{Type I}), $\delta_{11}=s^1_-(r)-s^1_-(r-\mu)$ (\textbf{Type II}),
$\delta_{12}=s^2_-(r-\mu)-s^2_-(r)$ (\textbf{Type I}), $\delta_{12}=s^1_-(r+\mu)-s^1_-(r)$ (\textbf{Type II}), $\delta_{21}=s_+(r)-s_+(r-\mu)$, $\delta_{22}=s_+(r+\mu)-s_+(r)$, $\eta_1=\eta_2=\mu$,  $R_1=|p_-|$, $R_2=|p_+|$, $\tilde R_1=-\sigma(|p_-|)$,
$\tilde R_2=\sigma(|p_+|)$, and $\theta\in[0,\pi/2]$ is the half of the angle between $p_+$ and $-p_-$.
\end{proof}

\subsubsection{\bf Final characterization of $R(F_0)$.}
We are now ready to establish the  result concerning  the essential structure of $R(F_0)$ for both types. For our purpose, it is sufficient to stick only  to the diagonal components of matrices in $R(F_0)$.

\begin{thm}\label{thm:main2}
Let $0<r<\sigma(s_+)$. Then there exists a number $\mu_r'>0$  with $0<r-\mu_r'<r+\mu_r'<\sigma(s_+)$    such that for any $0<\mu\le \mu_r'$, the set $\mathcal{S}=\mathcal{S}_{r-\mu,r+\mu}\subset\R^{n+n}$ in \emph{(\ref{rough-1})}  satisfies the following:
\begin{itemize}
\item[(i)] $\sup_{(p,\beta)\in\mathcal S}|p|\le s_+(r+\mu)<s_+$ and   $\,\sup_{(p,\beta)\in\mathcal S}|\beta|\le r+\mu<\sigma(s_+)$; hence $\mathcal{S}$ is bounded.

\item[(ii)] $\mathcal{S}$ is open.

\item[(iii)] For each $(p_0,\beta_0)\in\mathcal{S},$ there exist an open set $\mathcal{V}\subset\subset\mathcal{S}$ containing $(p_0,\beta_0)$ and $C^1$ functions $q:\bar{\mathcal{V}}\to\mathbb{S}^{n-1}$, $\gamma:\bar{\mathcal{V}}\to\R^n$, $t_\pm:\bar{\mathcal{V}}\to\R$ with $\gamma\cdot q=0$ and $t_-<0<t_+$  on $\bar{\mathcal{V}}$ such that for every $\xi=\begin{pmatrix} p & c\\ B & \beta\end{pmatrix}\in R(F_0)=R(F_{0,r-\mu,r+\mu})$ with $(p,\beta)\in\bar{\mathcal{V}}$, we have
\[
\xi+t_\pm\eta\in F_\pm=F_{\pm,r-\mu,r+\mu},
\]
where $t_\pm=t_\pm(p,\beta)$, $\eta=\begin{pmatrix} q(p,\beta) & b\\ \frac{1}{b}\gamma(p,\beta)\otimes q(p,\beta) & \gamma(p,\beta)\end{pmatrix}$, and $b\neq 0$ is arbitrary.
\end{itemize}
\end{thm}

\begin{proof}
Fix any $0<r<\sigma(s_+)$. First, we let $\mu>0$ be any number with $0<r-\mu <r+\mu <\sigma(s_+)$ and prove (i). Then we choose later an upper bound $\mu_r'$   of $\mu$ for the validity of (ii) and (iii) above.

We divide the proof into several steps.

1. To show that (i) holds, choose any $(p,\beta)\in\mathcal S$. By Lemma \ref{lem-rough}, $\xi:=\begin{pmatrix} p & 0\\ O & \beta\end{pmatrix}\in R(F_0)$, where $O$ is the $n\times n$ zero matrix. By the definition of $R(F_0)$, there exist two matrices $\xi_\pm=\begin{pmatrix} p_\pm & c_\pm\\ B_\pm & \sigma(|p_\pm|)\frac{p_\pm}{|p_\pm|}\end{pmatrix}\in F_\pm$ and a number $0<\lambda<1$ such that
$
\xi=\lambda\xi_++(1-\lambda)\xi_-.
$
So
\[
|p|=|\lambda p_++(1-\lambda)p_-|\le s_+(r+\mu),
\]
\[
|\beta|=\left|\lambda\sigma(|p_+|)\frac{p_+}{|p_+|}+(1-\lambda)\sigma(|p_-|) \frac{p_-}{|p_-|}\right|\le r+\mu \, ;
\]
%\[
%|(p,\beta)|=\left|\lambda(p_+,\sigma(|p_+|)\frac{p_+}{|p_+|})+(1-\lambda)(p_-,\sigma(|p_-|)\frac{p_-}{|p_-|})\right|
%\le  s_+(r_2) + r_2  ;
%\]
hence, $\sup_{(p,\beta)\in\mathcal S}|p|\le s_+(r+\mu)$ and  $\sup_{(p,\beta)\in\mathcal S}|\beta|\le r+\mu.$ Thus, $\mathcal S$ is bounded, and (i) is proved.
%\[
%\sup_{(p,\beta)\in\mathcal S}|(p,\beta)|\le\sqrt{(s_+(\bar\sigma))^2+\bar{\sigma}^2},\quad \sup_{(p,\beta)\in\mathcal S}|p|\le s_+(\bar\sigma).
%\]

2. We now turn to the remaining assertions that the set $\mathcal S=\mathcal S_{r-\mu,r+\mu}$ fulfills (ii) and (iii) for all sufficiently small $\mu>0$.  In this step, we still assume $\mu>0$ is any fixed number with $0<r-\mu <r+\mu <\sigma(s_+)$.

Let $(p_0,\beta_0)\in\mathcal S$.
Since $\xi_0:=\begin{pmatrix} p_0 & 0\\ O & \beta_0\end{pmatrix}\in R(F_0)$, it follows from Lemma \ref{lem-form} that there exist numbers $t_0<0<s_0$ and vectors $q_0,\,\gamma_0\in\R^n$ with $|q_0|=1$, $\gamma_0\cdot q_0=0$ such that $\xi_0+t_0\eta_0\in F_-$ and $\xi_0+s_0\eta_0\in F_+$, where $\eta_0=\begin{pmatrix} q_0 & b\\ \frac{1}{b}q_0\otimes\gamma_0 & \gamma_0\end{pmatrix}$ and $b\neq 0$ is any fixed number. Let $q'_0=t_0q_0\ne 0$, $\gamma'_0=t_0\gamma_0$, and $s'_0=s_0/t_0<0$; then
\begin{equation}\label{thm:main2-1}
\begin{cases}
\gamma'_0\cdot q'_0=0,\quad s_+(r-\mu)<|p_0+s'_0q'_0|<s_+(r+\mu),\\
s^2_-(r+\mu)<|p_0+q'_0|<s^2_-(r-\mu),\quad\mbox{(\textbf{Type I})}\\
s^1_-(r-\mu)<|p_0+q'_0|<s^1_-(r+\mu),\quad\mbox{(\textbf{Type II})}\\
\sigma(|p_0+s'_0q'_0|)\frac{p_0+s'_0q'_0}{|p_0+s'_0q'_0|}=\beta_0+s'_0\gamma'_0,\\
\sigma(|p_0+q'_0|)\frac{p_0+q'_0}{|p_0+q'_0|}=\beta_0+\gamma'_0.\end{cases}
\end{equation}
Observe also that
\begin{equation}\label{thm:main2-2}
\begin{split}
&s_0-t_0\ge |(p_0+s_0q_0)|-|(p_0+t_0q_0)|>s_+(r-\mu)-s^2_-(r-\mu),\;\;\mbox{(\textbf{Type I})}\\
&s_0-t_0\ge |(p_0+s_0q_0)|-|(p_0+t_0q_0)|>s_+(r-\mu)-s^1_-(r+\mu).\;\;\mbox{(\textbf{Type II})}
\end{split}
\end{equation}
%\[
%|\gamma_0|\le \frac{\bar\sigma+|\beta_0|}{\max\{t_0,-s_0\}}\le \frac{2\bar\sigma}{\frac{1}{2}(s_2-s_1)}\le\frac{4\sigma(s_1)}{s_2-s_1}.**
%\]

Next, consider the function $F$ defined by
\begin{equation*}
F(\gamma',q',s';p,\beta) = \begin{pmatrix}  \sigma(|p+s'q'|)\frac{p+s'q'}{|p+s'q'|}-\beta-s'\gamma'  \\
 \sigma(|p+q'|)\frac{p+q'}{|p+q'|}-\beta-\gamma' \\
  \gamma'\cdot q'\end{pmatrix} \in\R^{n+n+1}
\end{equation*}
for all $\gamma',\,q',\,p,\,\beta\in\R^n$ and $s'\in\R$ with  $s_+(r-\mu)<|p+s'q'|<s_+(r+\mu)$, $s^2_-(r+\mu)<|p+q'|<s^2_-(r-\mu)$ (\textbf{Type I}), $s^1_-(r-\mu)<|p+q'|<s^1_-(r+\mu)$ (\textbf{Type II}). Then  $F$ is $C^1$ in the described open subset of $\R^{n+n+1+n+n}$, and the observation (\ref{thm:main2-1}) yields that
\[
F(\gamma'_0,q'_0,s'_0;p_0,\beta_0)=0.%\in\R^{n+n+1}.
\]

Suppose for the moment that the Jacobian matrix $D_{(\gamma',q',s')}F$ %in $\mathbb{M}^{(n+n+1)\times(n+n+1)}$
is invertible at the point $(\gamma'_0,q'_0,s'_0;p_0,\beta_0)$; then the Implicit Function Theorem implies the following: There exist a bounded domain $\tilde{\mathcal{V}}=\tilde{\mathcal{V}}_{(p_0,\beta_0)}\subset\R^{n+n}$ containing $(p_0,\beta_0)$ and $C^1$ functions $\tilde\gamma,\,\tilde{q}\in\R^n$, $\tilde s\in\R$ of $(p,\beta)\in\tilde{\mathcal{V}}$ such that
\[
\tilde{\gamma}(p_0,\beta_0)=\gamma'_0,\;\; \tilde{q}(p_0,\beta_0)=q'_0,\;\; \tilde{s}(p_0,\beta_0)=s'_0
\]
and that
\[
\tilde s(p,\beta)<0,\;\; s_+(r-\mu)<|p+\tilde s(p,\beta)\tilde q(p,\beta)|<s_+(r+\mu),
\]
\[
s^2_-(r+\mu)<|p+\tilde q(p,\beta)|<s^2_-(r-\mu),\quad\mbox{(\textbf{Type I})}
\]
\[
s^1_-(r-\mu)<|p+\tilde q(p,\beta)|<s^1_-(r+\mu),\quad\mbox{(\textbf{Type II})}
\]
\[
F(\tilde\gamma(p,\beta),\tilde q(p,\beta),\tilde s(p,\beta);p,\beta)=0\quad \forall(p,\beta)\in\tilde{\mathcal{V}}.
\]
Define functions
\[
\gamma=-\frac{\tilde\gamma}{|\tilde q|}, \;\; q=-\frac{\tilde q}{|\tilde q|},\;\; t_+=-\tilde s |\tilde q|,\;\; t_-=-|\tilde q| \quad\mbox{in $\tilde{\mathcal{V}}$};
\]
then
\[
s_+(r-\mu)<|p+ t_+q|<s_+(r+\mu),
\]
\[
s^2_-(r+\mu)<|p+t_- q|<s^2_-(r-\mu),\quad\mbox{(\textbf{Type I})}
\]
\[
s^1_-(r-\mu)<|p+t_- q|<s^1_-(r+\mu),\quad\mbox{(\textbf{Type II})}
\]
\[
\sigma(|p+t_\pm q|)\frac{p+t_\pm q}{|p+t_\pm q|}=\beta+t_\pm \gamma,\;\; |q|=1,\;\; \gamma\cdot q=0,\;\; t_-<0<t_+,
\]
where $(p,\beta)\in\tilde{\mathcal{V}}$, $\gamma=\gamma(p,\beta)$, $q=q(p,\beta),$  and $t_\pm=t_\pm(p,\beta)$.

Let $(p,\beta)\in\tilde{\mathcal V}$, $B\in\mathbb{M}^{n\times n}$, $\tr B=0$, $b,c\in\R$, $b\neq 0$, $q=q(p,\beta)$, $\gamma=\gamma(p,\beta)$, $t_\pm=t_\pm(p,\beta)$, $\xi=\begin{pmatrix} p & c\\ B & \beta\end{pmatrix}$, and $\eta=\begin{pmatrix} q & b\\ \frac{1}{b}\gamma\otimes q & \gamma\end{pmatrix}$. Then $\xi_\pm:=\xi+t_\pm\eta\in F_\pm$. By the definition of $R(F_0)$, $\xi\in(\xi_-,\xi_+)\subset R(F_0)$. By Lemma \ref{lem-rough}, we thus have $(p,\beta)\in\mathcal{S}$; hence $\tilde{\mathcal{V}}_{(p_0,\beta_0)}=\tilde{\mathcal{V}}\subset\mathcal S$. This proves that $\mathcal S$ is open. Choosing any open set $\mathcal V\subset\subset\tilde{\mathcal{V}}$ with $(p_0,\beta_0)\in\mathcal{V}$,  the assertion  (iii) will hold.

3. In this step, we continue Step 2 to deduce an equivalent condition for the invertibility of the Jacobian matrix $D_{(\gamma',q',s')}F$  at  $(\gamma'_0,q'_0,s'_0;p_0,\beta_0)$. By direct computation,
\[
D_{(\gamma',q',s')}F= \begin{pmatrix} -s'I_n & M_{s'} & \omega^-_{s'} \\ -I_n & M_1 & 0  \\ q' & \gamma' & 0  \end{pmatrix}\in\mathbb{M}^{(n+n+1)\times(n+n+1)},
\]
where $I_n$ is the $n\times n$ identity matrix,
\[
M_{s'}=s'\left (\sigma'(|p+s'q'|)-\frac{\sigma(|p+s'q'|)}{|p+s'q'|}\right )\frac{p+s'q'}{|p+s'q'|}\otimes \frac{p+s'q'}{|p+s'q'|}+s'\frac{\sigma(|p+s'q'|)}{|p+s'q'|}I_n,
\]
\[
\omega^\pm_{s'}=\left (\sigma'(|p+s'q'|)-\frac{\sigma(|p+s'q'|)}{|p+s'q'|}\right )\left (\frac{p+s'q'}{|p+s'q'|}\cdot q'\right )\frac{p+s'q'}{|p+s'q'|}+\frac{\sigma(|p+s'q'|)}{|p+s'q'|}q'\pm\gamma'.
\]
Here the prime only in $\sigma'$ denotes the derivative.
For notational simplicity,  we write $(\gamma',q',s';p,\beta)=(\gamma'_0,q'_0,s'_0;p_0,\beta_0)$.
Applying suitable elementary row operations, as $s'<0$, we have
\[
D_{(\gamma',q',s')}F \,\to\, \begin{pmatrix} -s'I_n & M_{s'} & \omega^-_{s'} \\ O & M_1-\frac{1}{s'}M_{s'} & -\frac{1}{s'}\omega^-_{s'} \\ 0 & \gamma'+\frac{q'_1}{s'}(M_{s'})^1+\cdots+\frac{q'_n}{s'}(M_{s'})^n & \frac{1}{s'}q'\cdot\omega^-_{s'}\end{pmatrix}
\]
\[
\to\, \begin{pmatrix} -s'I_n & M_{s'} & \omega^-_{s'} \\ O & s'M_1-M_{s'} & -\omega^-_{s'} \\ 0 & s'\gamma'+q'_1(M_{s'})^1+\cdots+q'_n(M_{s'})^n & q'\cdot\omega^-_{s'}\end{pmatrix},
\]
where $O$ is the $n\times n$ zero matrix, and $(M_{s'})^i$ is the $i$th row of $M_{s'}$. Since $|q'|=-t_0$, $\gamma'\cdot q'=0$, and $s_+(r-\mu)<|p+s'q'|<s_+(r+\mu)$, we have %from Hypothesis (H) that
\[
q'\cdot\omega^-_{s'}=\left (\sigma'(|p+s'q'|)-\frac{\sigma(|p+s'q'|)}{|p+s'q'|}\right ) \left (\frac{p+s'q'}{|p+s'q'|}\cdot q'\right )^2+\frac{\sigma(|p+s'q'|)}{|p+s'q'|}t_0^2
\]
\[
=t_0^2\left(\cos^2\theta'\sigma'(|p+s'q'|)+(1-\cos^2\theta')\frac{\sigma(|p+s'q'|)}{|p+s'q'|}\right)>0,
\]
where $\theta'\in[0,\pi]$ is the angle between $p+s'q'$ and $q'$. %Observe here that the forward part of $\sigma$ in the definition of $F_+$ becomes essential to guarantee that $\sigma'(|p+s'q'|)>0$.
After some elementary column operations to the last matrix from the above row operations, we obtain
\[
D_{(\gamma',q',s')}F \,\to\,  \begin{pmatrix} -s'I_n & M_{s'}-N_{s'} & \omega^-_{s'} \\ O & s'M_1-M_{s'}+N_{s'} & -\omega^-_{s'} \\ 0 & 0 & q'\cdot\omega^-_{s'} \end{pmatrix},
\]
where the $j$th column of  $N_{s'}\in\mathbb{M}^{n\times n}$ is
$\frac{s'\gamma'_j+q'\cdot(M_{s'})_j}{q'\cdot\omega^-_{s'}}\omega^-_{s'}$. So $D_{(\gamma',q',s')}F$ is invertible if and only if the $n\times n$ matrix  $M_1-\frac{1}{s'}M_{s'}+\frac{1}{s'}N_{s'}$ is invertible. We compute
\[
M_1-\frac{1}{s'}M_{s'}+\frac{1}{s'}N_{s'} = \left(\sigma'(|p+q'|)-\frac{\sigma(|p+q'|)}{|p+q'|}\right)\frac{p+q'}{|p+q'|}\otimes \frac{p+q'}{|p+q'|}
\]
\[
+\frac{\sigma(|p+q'|)}{|p+q'|}I_n-\left(\sigma'(|p+s'q'|)-\frac{\sigma(|p+s'q'|)}{|p+s'q'|}\right )\frac{p+s'q'}{|p+s'q'|}\otimes \frac{p+s'q'}{|p+s'q'|}
-\frac{\sigma(|p+s'q'|)}{|p+s'q'|}I_n
\]
\[
+\frac{1}{q'\cdot\omega^-_{s'}}\omega^-_{s'}\otimes \left [\gamma'
+\left(\sigma'(|p+s'q'|)-\frac{\sigma(|p+s'q'|)}{|p+s'q'|}\right)\left(\frac{p +s'q'}{|p+s'q'|}\cdot q'\right)\frac{p+s'q'}{|p+s'q'|}+\frac{\sigma(|p+s'q'|)}{|p+s'q'|}q'\right]
\]
\[
\begin{split}
=(a_1-a_{s'})I_n & +(b_1-a_1)\frac{p+q'}{|p+q'|}\otimes\frac{p+q'}{|p+q'|}\\
& -(b_{s'}-a_{s'})
\frac{p+s'q'}{|p+s'q'|}\otimes\frac{p+s'q'}{|p+s'q'|}
+
\frac{1}{q'\cdot\omega^-_{s'}}\omega^-_{s'}\otimes\omega^+_{s'},
\end{split}
\]
where
\[
a_{l}=\frac{\sigma(|p+lq'|)}{|p+lq'|},\quad b_{l}=\sigma'(|p+lq'|)\quad\mbox{for $l=s',1$}.
\]
Since $a_1<0<a_{s'}$, we can set
\[
\begin{split} B=B_{(p_0,\beta_0)}=&\frac{1}{a_1-a_{s'}}(M_1-\frac{1}{s'}M_{s'}+\frac{1}{s'}N_{s'})\\
= &
I_n+\frac{b_1-a_1}{a_1-a_{s'}}\frac{p+q'}{|p+q'|}\otimes\frac{p+q'}{|p+q'|}\\
 &-\frac{b_{s'}-a_{s'}}{a_1-a_{s'}}
\frac{p+s'q'}{|p+s'q'|}\otimes\frac{p+s'q'}{|p+s'q'|} +
\frac{1}{(a_1-a_{s'})q'\cdot\omega^-_{s'}}\omega^-_{s'}\otimes\omega^+_{s'};\end{split}
\]
then
$D_{(\gamma',q',s')}F$ is invertible if and only if the matrix
$B\in\mathbb{M}^{n\times n}$ is invertible.

4. To close the arguments in Step 2 and thus to finish the proof, we choose a suitable $\mu_r'>0$ with $0<r-\mu_r'<r+\mu_r'<\sigma(s_+)$ in such a way that for each $0<\mu\le\mu_r'$, the matrix $B=B_{(p_0,\beta_0)}$, determined through Steps 2 and 3 for any given $(p_0,\beta_0)\in\mathcal S=\mathcal S_{r-\mu,r+\mu}$, is invertible.

First, let $0<\mu\le\mu_r'\le\mu_r$, where the number $\mu_r>0$ is determined by Theorem \ref{lem-inv}.
By Hypothesis (NF),
\[
\frac{\sigma(l)}{l}<0<\frac{\sigma(k)}{k}\quad\forall  k\in[s_+(r-\mu),s_+(r+\mu)],
\]
\[
\forall l\in[s^2_-(r+\mu),s^2_-(r-\mu)]\;\;\mbox{(\textbf{Type I})},\;\; \forall l\in[s^1_-(r-\mu),s^1_-(r+\mu)]\;\;\mbox{(\textbf{Type II})}.
\]
So we can define a real-valued continuous function (to express the determinant of the matrix $B=B_{(p_0,\beta_0)}$ from Step 3)
\[
\mbox{DET}(v,u,q,\gamma)=\det\Big(I_n+\frac{\sigma'(|u|)-\frac{\sigma(|u|)}{|u|}}
{\frac{\sigma(|u|)}{|u|}-\frac{\sigma(|v|)}{|v|}}\frac{u}{|u|}\otimes\frac{u}{|u|}
-\frac{\sigma'(|v|)-\frac{\sigma(|v|)}{|v|}}
{\frac{\sigma(|u|)}{|u|}-\frac{\sigma(|v|)}{|v|}}\frac{v}{|v|}\otimes\frac{v}{|v|}
\]
\[
+\frac{1}{(\frac{\sigma(|u|)}{|u|}-\frac{\sigma(|v|)}{|v|})((\sigma'(|v|)-\frac{\sigma(|v|)}{|v|})
(\frac{v}{|v|}\cdot q)^2 + \frac{\sigma(|v|)}{|v|})}
\big((\sigma'(|v|)-\frac{\sigma(|v|)}{|v|})(\frac{v}{|v|}\cdot q)\frac{v}{|v|}
\]
\[
+\frac{\sigma(|v|)}{|v|}q-\gamma \big)\otimes \big((\sigma'(|v|)-\frac{\sigma(|v|)}{|v|})(\frac{v}{|v|}\cdot q)\frac{v}{|v|}
+\frac{\sigma(|v|)}{|v|}q+\gamma \big)\Big)
\]
on the compact set $\mathcal M$ of points $(v,u,q,\gamma)\in\R^n\times\R^n\times\mathbb{S}^{n-1}\times\R^n$ with
\[
|v|\in[s_+(r-\mu),s_+(r+\mu)],\;\;|\gamma|\le 1,
\]
\[
|u|\in[s^2_-(r+\mu),s^2_-(r-\mu)]\;\;\mbox{(\textbf{Type I})},\;\;|u|\in[s^1_-(r-\mu),s^1_-(r+\mu)]\;\;\mbox{(\textbf{Type II})}.
\]
With $\bar k=s_+(r)$ and $\bar l=-s^2_-(r)$ (\textbf{Type I}), $\bar l=-s^1_-(r)$ (\textbf{Type II}), for each $q\in\mathbb{S}^{n-1},$
\[
\mbox{DET}(\bar k q,\bar lq,q,0)=\det\Big(I_n+\frac{\sigma'(-\bar l)-\frac{\sigma(-\bar l)}{-\bar l}+\frac{\sigma(\bar k)}{\bar k}}{\frac{\sigma(-\bar l)}{-\bar l}-\frac{\sigma(\bar k)}{\bar k}}q\otimes q \Big)\ne 0,
\]
since $\sigma'(-\bar l)\ne 0$ and hence the fraction in front of $q\otimes q$ is different from $-1$. So
\[
d:=\min_{q\in\mathbb{S}^{n-1}}|\mbox{DET}(\bar kq,\bar lq,q,0)|>0.
\]

Next, choose a number $\rho>0$ such that for all $(v,u,q,\gamma),(\tilde v,\tilde u,\tilde q,\tilde \gamma)\in\mathcal M$ with
$|v-\tilde v|,\,|u-\tilde u|,\,|q-\tilde q| ,\,|\gamma-\tilde\gamma|<\rho$, we have
\begin{equation}\label{thm:main2-3}
|\mbox{DET}(v,u,q,\gamma)-\mbox{DET}(\tilde v,\tilde u,\tilde q,\tilde \gamma)|<d/2.
\end{equation}
Let $\mu_r'\in(0,\mu_r]$ be sufficiently small so that for all $0<\mu\le\mu_r'$,
\[
\begin{split}
h\big( & s^2_-(r),s_+(r),r; s^2_-(r)-s^2_-(r+\mu),s^2_-(r-\mu)-s^2_-(r), \\
&  s_+(r)-s_+(r-\mu),s_+(r+\mu)-s_+(r),\mu,\mu\big)<\tau,\;\; (\textbf{Type I})
\end{split}
\]
\[
\begin{split}
h\big( & s^1_-(r),s_+(r),r; s^1_-(r)-s^1_-(r-\mu),s^1_-(r+\mu)-s^1_-(r), \\
&  s_+(r)-s_+(r-\mu),s_+(r+\mu)-s_+(r),\mu,\mu\big)<\tau,\;\; (\textbf{Type II})
\end{split}
\]
where $h$ is the function in Theorem \ref{lem-inv}, and
\[
\tau:=\min\{\rho,\rho(s_+(r-\mu_r)-s^2_-(r-\mu_r))/4\},\;\;\mbox{(\textbf{Type I})}
\]
\[
\tau:=\min\{\rho,\rho(s_+(r-\mu_r)-s^1_-(r+\mu_r))/4\}.\;\;\mbox{(\textbf{Type II})}
\]

Now, fix  any $\mu\in(0,\mu_r']$, and let $B=B_{(p_0,\beta_0)}$ be the $n\times n$ matrix determined through Steps 2 and 3 in terms of any given $(p_0,\beta_0)\in\mathcal{S}=\mathcal{S}_{r-\mu,r+\mu}$. Let $p_+=p_0+s_0 q_0$ and $p_-=p_0+t_0 q_0$ from Step 2; then $p_\pm$ and $A(p_\pm)$ fulfill the conditions in Theorem \ref{lem-inv}. So this theorem implies that there exists a vector $\zeta^0\in\mathbb{S}^{n-1}$ such that
\[
\max\{|p^0_--p_-|,|p^0_+-p_+|,|A(p^0_-)-A(p_-)|,|A(p^0_+)-A(p_+)|\}<\tau,
\]
where $p^0_+=\bar{k}\zeta^0$, $p^0_-=\bar{l}\zeta^0$, and $A(p_\pm^0)=r\zeta^0$. Using (\ref{thm:main2-1}) and (\ref{thm:main2-2}),
\[
|p_+ -\bar{k}\zeta^0|<\rho,\;\;|p_- -\bar{l}\zeta^0|<\rho,
\]
\[
|q_0-\zeta^0|=|\frac{p_+-p_-}{s_0-t_0}-\zeta^0|\le\frac{|(p_+-p_-)-(\bar{k}-\bar{l})\zeta^0|+|(\bar{k}-\bar{l})-(s_0-t_0)|}{s_0-t_0}
\]
\[
\le \frac{2\tau+||p_+^0-p_-^0|-|p_+-p_-||}{s_0-t_0}<\frac{4\tau}{s_0-t_0}<\rho,
\]
\[
|\gamma_0|=|\frac{A(p_+)-A(p_-)}{s_0-t_0}|\le\frac{|A(p_+)-A(p_+^0)|+|A(p_-^0)-A(p_-)|}{s_0-t_0}<\rho.
\]
Since $\det(B)=\mbox{DET}(p_+,p_-,q_0,\gamma_0)$ and $|\mbox{DET}(\bar{k}\zeta^0,\bar{l}\zeta^0,\zeta^0,0)|\ge d$, it  follows from (\ref{thm:main2-3}) that
\[
|\det(B)|>d/2>0.
\]

The proof is now complete.
\end{proof}

\subsection{Relaxation of  $\nabla \omega(z)\in K_0$}

The following result is important for the convex integration with a linear constraint; its proof can be found in \cite[Lemma 4.5]{KY1}.

\begin{lem}\label{approx-lem} Let $\lambda_1,\lambda_2>0$ and $\eta_1=-\lambda_1\eta, \; \eta_2=\lambda_2\eta$ with
\[
\eta=\begin{pmatrix} q &  b\\\frac{1}{b}\gamma\otimes q & \gamma\end{pmatrix},\quad |q|=1,\; \gamma\cdot q=0,\;   b\ne 0.
\]
Let $G\subset \R^{n+1}$ be a bounded domain. Then for each $\epsilon>0$, there exists a function $\omega=(\varphi,\psi)\in C_c^{\infty}(\R^{n+1};\R^{1+n})$ with $\mathrm{supp}(\omega)\subset\subset G$  that satisfies the following properties:

(a) \; $\dv \psi =0$  in $G$,

(b) \; $|\{z\in G\;|\; \nabla\omega(z)\notin \{\eta_1,\;\eta_2\}\}|<\epsilon,$

(c) \; $\dist(\nabla \omega(z),[\eta_1,\eta_2])<\epsilon$ for all $z\in G,$

(d) \; $\|\omega\|_{L^\infty(G)}<\epsilon,$

(e) \; $\int_{\R^n}\varphi(x,t)\,dx=0$ for all $t\in\R$.

\end{lem}

We now state  the relaxation theorem for  homogeneous differential inclusion $\nabla \omega(z)\in K_0$  in a form that is suitable for a later use; we  restrict the inclusion
 only to  the diagonal components.

\begin{thm}\label{main-lemma}  Let $0<r <\sigma(s_+)$, and let $0<\mu\le\mu_r'$ for some number $\mu_r'>0$ with $0<r-\mu_r'<r+\mu_r'<\sigma(s_+)$ from Theorem \ref{thm:main2}. Let  $\mathcal K$ be a compact subset of\, $\mathcal S=\mathcal S_{r-\mu,r+\mu}$, and let $\tilde Q\times \tilde I$ be  a  box  in $\R^{n+1}.$  Then given any $\epsilon>0$, there exists a $\delta>0$ such that for each box $Q\times I\subset \tilde Q\times \tilde I$, point  $(p,\beta)\in\mathcal K$,  and  number $\rho>0$ sufficiently small, there exists a function $\omega=(\varphi,\psi)\in C^\infty_{c} (Q\times I;\R^{1+n})$  satisfying the following properties:

(a) \; $\dv \psi =0$  in $Q\times I$,

(b) \; $(p'+D\varphi(z), \beta'+ \psi_t(z))\in \mathcal{S}$ for all $z \in Q\times I$ and $|(p',\beta')-(p,\beta)|\leq \delta,$

(c) \; $\|\omega\|_{L^\infty(Q\times I)}<\rho,$

(d) \; $\int_{Q\times I} |\beta+\psi_t(z)-A(p+D\varphi(z))|dz<\epsilon {|Q\times I|}/{|\tilde Q\times \tilde I|},$

(e) \; $\int_{Q\times I} \mathrm{dist}((p+D\varphi(z),\beta+\psi_t(z)),\mathcal A)dz<\epsilon {|Q\times I|}/{|\tilde Q\times \tilde I|},$

(f) \; $\int_{Q} \varphi(x,t)dx=0$  for all $t\in I,$

(g) \; $\|\varphi_t\|_{L^\infty(Q\times I)}<\rho,$

\noindent where $\mathcal A=\mathcal A_{r-\mu,r+\mu}\subset \R^{n+n}$ is the set defined by
\[
\mathcal A = \big\{(p',A(p')) \,|\, |p'|\in [s^2_-(r+\mu),s^2_-(r-\mu)]\cup[s_+(r-\mu),s_+(r+\mu)] \big\},\;(\textbf{\emph{Type I}})
\]
\[
\mathcal A = \big\{(p',A(p')) \,|\, |p'|\in [s^1_-(r-\mu),s^1_-(r+\mu)]\cup[s_+(r-\mu),s_+(r+\mu)] \big\}.\;(\textbf{\emph{Type II}})
\]
\end{thm}

\begin{proof}
By Theorem \ref{thm:main2}, there exist finitely many open balls $\mathcal{B}_1,\cdots,\mathcal{B}_N\subset\subset\mathcal S$ covering $\mathcal K$ and $C^1$ functions $q_i:\bar{\mathcal{B}}_i\to \mathbb{S}^{n-1}$, $\gamma_i:\bar{\mathcal{B}}_i\to \R^{n}$, $t_{i,\pm}:\bar{\mathcal{B}}_i\to \R$ $(1\le i\le N)$ with $\gamma_i\cdot q_i=0$ and $t_{i,-}<0<t_{i,+}$ on $\bar{\mathcal B}_i$  such that for each $\xi=\begin{pmatrix} p & c \\ B & \beta\end{pmatrix}\in R(F_0)=R(F_{0,r-\mu,r+\mu})$ with $(p,\beta)\in\bar{\mathcal{B}}_i$, we have
\[
\xi+t_{i,\pm} \eta_i\in F_\pm=F_{\pm,r-\mu,r+\mu},
\]
where $t_{i,\pm}=t_{i,\pm}(p,\beta)$, $\eta_i=\begin{pmatrix} q_i(p,\beta) & b \\ \frac{1}{b}\gamma_i(p,\beta)\otimes q_i(p,\beta) & \gamma_i(p,\beta)\end{pmatrix}$, and $b\neq 0$ is arbitrary.

Let $1\le i\le N$. We write $\xi_i=\xi_i(p,\beta)=\begin{pmatrix} p & 0\\O & \beta\end{pmatrix}\in R(F_0)$ for $(p,\beta)\in\bar{\mathcal B}_{i}\subset\mathcal S$, where $O$ is  the $n\times  n$ zero matrix. We omit the dependence on $(p,\beta)\in\bar{\mathcal B}_{i}$ in the following whenever it is clear from the context. Given any $\rho>0$, we choose a constant $b_i=b_{i,\rho}$ with
\[
0<b_i<\min_{\bar{\mathcal{B}}_{i}} \frac{\rho}{t_{i,+}-t_{i,-}}.
\]
With this choice of  $b=b_i$, let $\eta_{i}$ be defined on $\bar{\mathcal{B}}_i$ as above. Then
\[
\xi_{i,\pm}=\begin{pmatrix} p_{i,\pm} & c_{i,\pm}\\B_{i,\pm} & \beta_{i,\pm}\end{pmatrix}:=\xi_i+t_{i,\pm}\eta_{i}\in F_\pm,
\]
\[
\xi_i=\lambda_i \xi_{i,+} + (1-\lambda_i)\xi_{i,-},\quad\lambda_i=\frac{-t_{i,-}}{t_{i,+}-t_{i,-}}\in(0,1)\quad \textrm{on}\;\;\bar{\mathcal{B}}_{i}.
\]
By the definition of $R(F_0)$,  on $\bar{\mathcal B}_i$,  both $\xi_{i,-}^\tau=\tau \xi_{i,+} + (1-\tau)\xi_{i,-}$ and $\xi_{i,+}^\tau=(1-\tau) \xi_{i,+} +  \tau \xi_{i,-}$ belong to $R(F_0)$ for all $\tau\in (0,1)$.  Let $0<\tau<\min_{1\le j\le N}\min_{\bar{\mathcal B}_j}\min\{\lambda_j,1-\lambda_j\}\le \frac12$ be a small number to be selected later.   Let $\lambda'_i=\frac{\lambda_i-\tau}{1-2\tau}$ on $\bar{\mathcal B}_i$. Then $\lambda'_i\in (0,1)$ and
$\xi_i=\lambda'_i\xi_{i,+}^\tau+(1-\lambda'_i)\xi_{i,-}^\tau$ on $\bar{\mathcal B}_i$.
Moreover, on $\bar{\mathcal B}_i$, $\xi_{i,+}^\tau-\xi_{i,-}^\tau=(1-2\tau)(\xi_{i,+} -\xi_{i,-} )$ is rank-one, $[\xi_{i,-}^\tau,\xi_{i,+}^\tau]\subset(\xi_{i,-},\xi_{i,+})\subset R(F_0)$, and
\[
c\tau\leq|\xi_{i,+}^\tau-\xi_{i,+}|=|\xi_{i,-}^\tau-\xi_{i,-}|=\tau |\xi_{i,+} -\xi_{i,-}|=\tau(t_{i,+}-t_{i,-})|\eta_i|\leq C\tau,
\]
where $C=\max_{1\le j\le N}\max_{\bar{\mathcal B}_j}(t_{j,+}-t_{j,-})|\eta_j|\geq\min_{1\le j\le N}\min_{\bar{\mathcal B}_j}(t_{j,+}-t_{j,-})|\eta_j|$ $=c>0.$ By continuity, $H_\tau=\bigcup_{(p,\beta)\in\bar{\mathcal{B}}_j,1\le j\le N}[\xi_{j,-}^\tau(p,\beta),\xi_{j,+}^\tau(p,\beta)]$ is a compact subset of $R(F_0)$, where $R(F_0)$ is open in the space
\[
\Sigma_0=\left\{\begin{pmatrix} p & c\\B & \beta\end{pmatrix}\;\Big|\; \mathrm{tr}B=0\right\},
\]
by Lemma \ref{lem-rough} and Theorem \ref{thm:main2}.
So $d_\tau=\mathrm{dist}(H_\tau,\partial|_{\Sigma_0}R(F_0))>0$, where $\partial|_{\Sigma_0}$ is the relative boundary in  $\Sigma_0$.

Let $\eta_{i,1}=-\lambda_{i,1}\eta_i=-\lambda'_i(1-2\tau)(t_{i,+}-t_{i,-})\eta_i,\,\eta_{i,2}= \lambda_{i,2}\eta_i=(1-\lambda'_i)(1-2\tau)(t_{i,+}-t_{i,-})\eta_i$ on $\bar{\mathcal B}_i$, where $\lambda_{i,1}=\tau(-t_{i,+})+(1-\tau)(-t_{i,-})>0,\,\lambda_{i,2}=(1-\tau)t_{i,+}+\tau t_{i,-}>0$  on $\bar{\mathcal B}_i$, and $\tau>0$ is so small that
\[
\min_{1\le j\le N}\min_{\bar{\mathcal B}_j}\lambda_{j,k}>0\;\;(k=1,\,2).
\]
Applying  Lemma \ref{approx-lem} to matrices $\eta_{i,1}=\eta_{i,1}(p,\beta),\,\eta_{i,2}=\eta_{i,2}(p,\beta)$ (also depending on $\rho$) for a fixed $(p,\beta)\in\bar{\mathcal{B}}_i$ and to a given box $G= Q\times I$, we obtain that for each $\rho>0$, there exist
a function $\omega=(\varphi,\psi)\in C^\infty_c( Q\times I;\R^{1+n})$  and an open set $G_\rho\subset\subset  Q\times I$ satisfying the following conditions:
\begin{equation}\label{approx-1}
\begin{cases} \mbox{(1) \;  $\dv \psi =0$  in $Q\times I$,}\\
\mbox{(2) \; $|(Q\times I) \setminus G_\rho|<\rho$;\; $\xi_i+\nabla  \omega(z)\in \{\xi^\tau_{i,-},\;\xi_{i,+}^\tau\}$ for all $z\in  G_\rho$,}\\
\mbox{(3) \; $\xi_i+\nabla  \omega(z)\in [\xi_{i,-}^\tau,\xi_{i,+}^\tau]_\rho$ for all $z\in Q\times I,$}\\
\mbox{(4) \; $\|\omega\|_{L^\infty(Q\times I)}<\rho$,} \\
\mbox{(5) \; $\int_Q \varphi(x,t)\,dx=0$ for all $t\in I$,} \\
\mbox{(6) \;  $\|\varphi_t\|_{L^\infty(Q\times I)}<2\rho,$}
\end{cases}
\end{equation}
where  $[\xi_{i,-}^\tau,\xi_{i,+}^\tau]_\rho$ denotes the $\rho$-neighborhood of the closed line segment $[\xi_{i,-}^\tau,\xi_{i,+}^\tau].$
Here, from (\ref{approx-1}.3), condition (\ref{approx-1}.6) follows as
\[
|\varphi_t|<|c_{i,+}-c_{i,-}|+\rho=(t_{i,+}-t_{i,-})|b_i|+\rho<2\rho \quad\textrm{in $Q\times I$.}
\]

Note (a), (c), (f), and (g) follow from (\ref{approx-1}), where $2\rho$ in  (\ref{approx-1}.6) can be adjusted to $\rho$ as in (g).
By the uniform continuity of $A$ on the set $J=\{p'\in\R^n\;|\; |p'|\le s_+ \}$,
we can find a $\delta'>0$ such that $|A(p')-A(p'')|< \frac{\epsilon}{3|\tilde Q\times\tilde I|}$ whenever $p',\,p''\in J$ and $|p'-p''|<\delta'.$
We then choose a $\tau>0$ so small that
\[
C\tau<\delta',\;\;C|\tilde Q\times \tilde I|\tau<\frac{\epsilon}{3}.
\]
Next, we  choose a $\delta>0$ such that
$
\delta<\frac{d_\tau}{2}.
$
If $0<\rho<\delta$, then by (\ref{approx-1}.1) and (\ref{approx-1}.3),
for all $z\in Q\times I$ and $|(p',\beta')-(p,\beta)|\le \delta$,
\[
\xi_i(p',\beta')+\nabla\omega(z)\in\Sigma_0,\quad\mathrm{dist}(\xi_i(p',\beta')+\nabla\omega(z),H_\tau)<
d_\tau,
\]
and so $\xi_i(p',\beta')+\nabla\omega(z)\in R(F_0),$ that is, $(p'+D\varphi(z), \beta'+ \psi_t(z))\in \mathcal{S}$. Thus (b) holds for all $0<\rho<\delta.$ In particular, $(p+D\varphi(z), \beta+ \psi_t(z))\in \mathcal{S}$ and so $|p+D\varphi(z)|\le s_+(r+\mu)$  and $|\beta+\psi_t(z)|\le r+\mu$ for all $z\in Q\times I$, by (i) of Theorem \ref{thm:main2}. Thus
\begin{equation*}
\begin{split}
\int_{ Q\times I} & |\beta+\psi_t -A(p+D\varphi)|dz \\
 &\le \int_{G_\rho}|\beta+\psi_t -A(p+D\varphi)|dz + 2\sigma(s_+)\rho
\\
&\le |Q\times I| \max\{|\beta_{i,\pm}^\tau -A(p_{i,\pm}^\tau)|\} +2\sigma(s_+)\rho
\\
&\le C|Q\times I|\tau+ |Q\times I| \max\{|A(p_{i,\pm}) -A(p_{i,\pm}^\tau)|\} +2\sigma(s_+)\rho \\
& \le \frac{2\epsilon|Q\times I|}{3|\tilde Q\times\tilde I|}+2\sigma(s_+)\rho,
\end{split}
\end{equation*}
where  $\xi^\tau_{i,\pm}=\begin{pmatrix} p^\tau_{i,\pm} & c^\tau_{i,\pm}\\B^\tau_{i,\pm} & \beta^\tau_{i,\pm}\end{pmatrix}$.
Thus, (d) holds for all $\rho>0$ satisfying $2\sigma(s_+)\rho<\frac{\epsilon|Q\times I|}{3|\tilde Q\times\tilde I|}$. Similarly,
\begin{equation*}
\begin{split}
\int_{Q\times I} & \mathrm{dist}((p+D\varphi(z),\beta+\psi_t(z)),\mathcal A)dz \\
 &\le \int_{G_\rho}\max|(p_{i,\pm}^\tau,\beta_{i,\pm}^\tau)-(p_{i,\pm},\beta_{i,\pm})|dz + 2(s_+ +\sigma(s_+))\rho
\\
&\le C|Q\times I|\tau+ 2(s_+ +\sigma(s_+))\rho \\
& \le \frac{\epsilon|Q\times I|}{3|\tilde Q\times\tilde I|}+2(s_+ +\sigma(s_+))\rho;
\end{split}
\end{equation*}
therefore, (e)  holds for all  $\rho>0$ with $(s_++\sigma(s_+))\rho<\frac{\epsilon|Q\times I|}{3|\tilde Q\times\tilde I|}$.

We have verified (a) -- (g) for any $(p,\beta)\in\bar{\mathcal{B}}_i$ and $1\le i\le N$, where $\delta>0$ is independent of the index $i$. Since $\mathcal B_1,\cdots, \mathcal B_N$ cover $\mathcal K$,
the proof is now complete.
\end{proof}

\section{Boundary function $\Phi$ and the admissible set $\mathcal U$ \\by a countable open covering}\label{sec:add-set}

To start the proof of Theorem \ref{thm:NF-1}, assume $\Omega$ and $u_0$ satisfy (\ref{assume-1}) and (\ref{av-0}).

\subsection{Boundary function $\Phi$} We first construct  a suitable boundary function $\Phi=(u^*,v^*)$ to prove Theorem \ref{thm:NF-1} in the setting of the general existence theorem, Theorem \ref{thm:main}. Assuming all the hypotheses in Theorem \ref{thm:NF-1}, we fix any $\tilde r\in(\sigma(m_0'),\sigma(s_+)).$ For each $r\in(0,\tilde r)$, let $\bar\mu_r>0$ be chosen so that
\[
0<r-\bar\mu_r<r+\bar\mu_r<\tilde r \quad\mbox{and}\quad \bar\mu_r\le \mu_r',
\]
where $\mu_r'>0$ is some number from Theorem \ref{main-lemma}. Then $\{I_r:=(r-\bar\mu_r,r+\bar\mu_r)\}_{r\in(0,\tilde r)}$ is an  open covering for the interval $(0,\tilde r)$. For convenience, we select a countable sub-covering $\{I_{r_k}\}_{k\in\N}$ of $\{I_r\}_{r\in(0,\tilde r)}$ for $(0,\tilde r)$. We now define a \emph{diagonal-covering set} $\mathcal S_{dc}\subset\R^{n+n}$ by
\[
\mathcal S_{dc}=\bigcup_{k\in\N}\mathcal S_{r_k-\bar\mu_{r_k},r_k+\bar\mu_{r_k}};
\]
then by Theorem \ref{thm:main2}, $\mathcal S_{dc}\subset\R^{n+n}$ is open and bounded.

Next, we apply Lemma \ref{lem:modi-NF} to the number $r=\tilde r$ in order  to determine functions   $\tilde\sigma,\,\tilde f\in C^{1+\alpha}([0,\infty))$   satisfying its conclusion. Also, let $\tilde A(p)=\tilde f(|p|^2)p$ $(p\in\R^n).$ Then the following holds.

\begin{lem}\label{lem-a} We have
\[
(p,\tilde A(p))\in \mathcal S_{dc}  \quad\forall\; 0<|p|<s_+(\tilde r).
\]
\end{lem}
\begin{proof}
Let $0<s=|p|<s_+(\tilde r)$, $r=\tilde\sigma(s)$ and $\zeta=p/|p|$, so that  $\zeta\in \mathbb{S}^{n-1}$, $\tilde A(p)=r\zeta$, and $0<r<\tilde r$. Set $p_{+}=s_+(r)\zeta$, $p_{-}=-s^2_-(r)\zeta$ (\textbf{Type I}), $p_{-}=-s^1_-(r)\zeta$ (\textbf{Type II}), and $\beta_{\pm}=r\zeta$. Then $A(p_{\pm})=r\zeta=\beta_{\pm}$. Define
$
\xi=\begin{pmatrix} p & 0\\ O & \tilde A(p)\end{pmatrix}
$
and
$
\xi_\pm=\begin{pmatrix} p_\pm & 0\\ O & \beta_\pm\end{pmatrix}.
$
Then $\xi=\lambda\xi_++(1-\lambda)\xi_-$ for some $0<\lambda<1.$

Observe now that $r\in I_{r_k}$ for some $k\in\N$, that is, $r_k-\bar\mu_{r_k}<r<r_k+\bar\mu_{r_k}$. We thus have
\[
\xi_\pm\in F_{\pm,r_k-\bar\mu_{r_k},r_k+\bar\mu_{r_k}}.
\]
Since  $\mathrm{rank}(\xi_+-\xi_-)=1$, it follows from the definition of $R(F_{0,r_k-\bar\mu_{r_k},r_k+\bar\mu_{r_k}})$ that $\xi\in(\xi_-,\xi_+)\subset R(F_{0,r_k-\bar\mu_{r_k},r_k+\bar\mu_{r_k}})$. Thus by Lemma \ref{lem-rough},
\[
(p,\tilde A(p))\in\mathcal S_{r_k-\bar\mu_{r_k},r_k+\bar\mu_{r_k}}\subset \mathcal S_{dc}.
\]
\end{proof}

By Lemma \ref{lem:modi-NF}, equation $u_t=\dv (\tilde A(Du))$ is uniformly parabolic. So by Theorem \ref{existence-gr-max}, the  initial-Neumann boundary value problem
\begin{equation}\label{ib-para}
\begin{cases} u^*_t =\dv (\tilde A(Du^*))&\mbox{in $\Omega_T$}\\
\partial u^*/\partial \n  =0 & \mbox{on $\partial \Omega\times (0,T)$}  \\
u^*(x,0)=u_0(x), &x\in \Omega
\end{cases}
\end{equation}
admits a unique classical solution  $u^*\in C^{2+\alpha,1+\alpha/2}(\bar\Omega_T)$.

From conditions (\ref{assume-1}) and (\ref{av-0}),  we can find a function $h\in C^{2+\alpha}(\bar\Omega)$ satisfying
\[
\Delta h=u_0\;\;\mbox{in}\;\;\Omega \quad\mbox{and}\quad \partial h/\partial \n=0\;\;\mbox{on}\;\;\partial\Omega.
\]
Let $v_0=Dh\in C^{1+\alpha}(\bar\Omega;\R^n)$ and define,  for $(x,t)\in\Omega_T$,
\begin{equation}\label{def-v}
v^*(x,t)=v_0(x)+\int_0^t \tilde A(Du^*(x,s))\,ds.
\end{equation}
Then it is easily seen that $\Phi := (u^*,v^*)\in C^{1}(\bar\Omega_T;\R^{1+n})$ satisfies (\ref{bdry-1}); that is,
\begin{equation}\label{bdry-2}
\begin{cases} u^*(x,0)=u_0(x) \; (x\in\Omega),\\
\dv v^*=u^*\;\;\textrm{in $\Omega_T$}, \\
v^*(\cdot,t) \cdot \n|_{\partial\Omega} =0 \; \; \forall\; t\in [0,T],\end{cases}
\end{equation}
and so $\Phi$ is a boundary function for the initial datum $u_0$.

Next, let
\[
\mathcal F=\left\{(p,A(p))\;|\; |p|\in\{0\}\cup[s_+(\tilde r),\max\{M,s_+\}] \right\},
\]
where $M=\|Du^*\|_{L^\infty(\Omega_T)}$.
Then we have the following:

\begin{lem} \label{lem-b}
\[
(Du^*(x,t), v^*_t(x,t))\in \mathcal S_{dc} \cup \mathcal F\quad \forall\; (x,t)\in \Omega_T.
\]
\end{lem}
\begin{proof}
Let $(x,t)\in\Omega_T$ and $p=Du^*(x,t)$.

First, assume $|p|\in\{0\}\cup[s_+(\tilde r),\max\{M,s_+\}]$. Then $\tilde A(p)=A(p)$ and hence by (\ref{def-v})
\[
(Du^*(x,t),v_t^*(x,t))=(p,\tilde A(p))=(p,A(p))\in \mathcal F.
\]
Otherwise, we have $0<|p|<s_+(\tilde r)$, and so by Lemma \ref{lem-a} and (\ref{def-v}),
\[
(Du^*(x,t),v_t^*(x,t))=(p,\tilde A(p))\in \mathcal S_{dc}.
\]
\end{proof}

We adopt the following terminology that is needed below.
\begin{defn} Let $G$ be an open set in $\R^N$.
 We say  a function $u$  is {\em piecewise $C^1$} in $G$ and write  $u\in C^1_{piece}(G)$  if there exists a sequence of disjoint open sets $\{G_j\}_{j=1}^\infty$ in $G$ such that
\[
u\in C^1(\bar G_j)  \;\; \forall j\in\mathbb{N},\quad |G\setminus \cup_{j=1}^\infty G_j|=0.
\]
It is then necessary to have   $|\partial G_j|=0\;\forall j\in\mathbb{N}$.
%(For our purpose it is also acceptable to allow only \emph{finitely}  many pieces  in this definition.)
\end{defn}

\subsection{Selection of interface}  To separate the space-time domain $\Omega_T$ into the classical and micro-oscillatory parts for Lipschitz solutions, we set
\[
\begin{split}
\Omega_T^0 & =\{(x,t)\in\Omega_T\,|\,|Du^*(x,t)|=0\},\\
\Omega_T^1 & =\{(x,t)\in\Omega_T\,|\,0<|Du^*(x,t)|<s_+(\tilde r)\},\\
\Omega_T^2 & =\{(x,t)\in\Omega_T\,|\,|Du^*(x,t)|=s_+(\tilde r)\},\\
\Omega_T^{\tilde r}=\Omega_T^3 & =\{(x,t)\in\Omega_T\,|\,|Du^*(x,t)|>s_+(\tilde r)\},\\
\Omega_0^{\tilde r} & =\{(x,0) \,|\, x\in\Omega,\,|Du_0(x)|>s_+(\tilde r) \},
\end{split}
\]
then $\Omega_T=\cup_{k=0}^3\Omega_T^k$, $\Omega_T^1\ne\emptyset$, and $\Omega_0^{\tilde r}\subset\partial\Omega_T^{\tilde r}$. Observe from the proof of Lemma \ref{lem-b} that $(Du^*,v_t^*)\in\mathcal S_{dc}$ in $\Omega_T^1$.

\subsection{The admissible set $\mathcal U$} Let $m=\|u_t^*\|_{L^\infty(\Omega_T)}+1.$  We define $\mathcal U$ to be the set of all $u\in W_{u^*}^{1,\infty}(\Omega_T)$ satisfying
\begin{equation}\label{ad-class-NF}
\left\{
\begin{array}{l}
  \mbox{$\|u_t\|_{L^\infty(\Omega_T)}<m,$} \\
  \mbox{$\exists$an open set $G\subset\subset\Omega_T^1$ with $|\partial G|=0$ such that} \\
  \mbox{\quad$u\equiv u^*$ in $\Omega_T\setminus G$ and $u\in C^1_{piece}(G)$, and} \\
  \mbox{$\exists \, v\in W_{v^*}^{1,\infty}(\Omega_T;\R^n)$ such that $v\equiv v^*$ in $\Omega_T\setminus G$,} \\
  \mbox{\quad$v\in C^1_{piece}(G;\R^n)$, $\dv v=u$ and $(Du,v_t)\in \mathcal S_{dc}\cup \mathcal F$ a.e.\;in $\Omega_T$,} \\
  \mbox{\quad and $(Du,v_t)\in \mathcal S_{dc}$ a.e.\;in $\Omega_T^1$.}
\end{array}\right.
\end{equation}
Next, for each $\epsilon>0$, let $\mathcal U_\epsilon$ be the set of all $u\in\mathcal U$ satisfying, in addition to (\ref{ad-class-NF}),
\[\left\{
\begin{array}{l}
  \int_{\Omega_T} |v_t-A(Du)|dxdt\le\epsilon|\Omega_T|, \\
  \int_{\Omega_T^1} \mathrm{dist}((Du,v_t),\mathcal B) dxdt\le\epsilon|\Omega_T^1|,
\end{array}\right.
\]
where $\mathcal B=\mathcal B_{\tilde r}\subset \R^{n+n}$ is the set given by
\[
\mathcal B=\{(p,A(p)) \,|\, |p|\in[s^2_-(\tilde r),s_+(\tilde r)]\},\quad (\textbf{Type I})
\]
\[
\mathcal B=\{(p,A(p)) \,|\, |p|\in[0,s^1_-(\tilde r)]\cup[s_0,s_+(\tilde r)]\}.\quad (\textbf{Type II})
\]
Observe here that
\[
\bigcup_{k\in\N} \mathcal A_{r_k-\bar\mu_{r_k},r_k+\bar\mu_{r_k}}\subset \mathcal B,
\]
where the sets $\mathcal A_{r_k-\bar\mu_{r_k},r_k+\bar\mu_{r_k}}$ are as in Theorem \ref{main-lemma}.

\begin{remk}\label{rmk-1} Summarizing the above, we have constructed a boundary function $\Phi=(u^*,v^*)\in C^1(\bar\Omega_T;\R^{1+n})$ for the initial datum $u_0$ in such a way that the admissible set $\mathcal U$ contains $u^*$; so $\mathcal U$ is nonempty. Also  $\mathcal U$ is a bounded subset of $W^{1,\infty}_{u^*}(\Omega_T)$, since $\mathcal S_{dc}\cup \mathcal F$ is bounded and $\|u_t\|_{L^\infty(\Omega_T)}<m$ for all $u\in\mathcal U$. Moreover, by (i) of Theorem \ref{thm:main2} and the definition of $\mathcal F$, for each $u\in\mathcal U$, its corresponding vector function $v$ satisfies $\|v_t\|_{L^\infty(\Omega_T)}\le \max\{\sigma(M_0),\sigma(s_+)\}$; this bound  plays the role of a fixed number $R>0$ in the general density approach in Subsection \ref{subsec-approach}. Finally, note that $\tilde A(Du^*)\ne A(Du^*)$ in the nonempty open set $\Omega_T^1$;  hence $u^*$ itself is not a Lipschitz solution to (\ref{ib-P}).

In view of the general existence theorem, Theorem \ref{thm:main}, it only remains to prove the $L^\infty$-density of $\mathcal U_\epsilon$ in $\mathcal U$ towards the existence of infinitely many Lipschitz solutions to problem (\ref{ib-P}) for both types; this core subject is carried out in the next section.
\end{remk}

\section{Density of $\mathcal U_\epsilon$ in $\mathcal U$: \\Final step for the proof  of Theorem \ref{thm:NF-1}}\label{sec:den-proof}

In this section, we follow Section \ref{sec:add-set} to complete the proof of Theorem \ref{thm:NF-1}.

\subsection{The density property} The density theorem below is the last preparation  for both \textbf{Types I} and \textbf{II}.

\begin{thm} \label{thm-density-1} For each $\epsilon>0$, $\mathcal U_\epsilon$ is dense in $\mathcal U$ under the $L^\infty$-norm.
\end{thm}

\begin{proof}%[Proof of Theorem \ref{thm-density-1}]

Given any $\epsilon>0$, let  $u\in \mathcal U$ and $\eta>0$. The goal is to construct a function  $\tilde u\in  \mathcal U_\epsilon$ such that $\|\tilde u-u\|_{L^\infty(\Omega_T)}<\eta.$
For clarity, we divide the proof into several steps.

1. Note from (\ref{ad-class-NF}) that $\|u_t\|_{L^\infty(\Omega_T)}<m-\bar\tau_0$ for some $\bar\tau_0>0$, there exists an open set $G\subset\subset\Omega_T^1$ with $|\partial G|=0$ such that $u\equiv u^*$  in  $\Omega_T\setminus G$ and $u\in C^1_{piece}(G)$,  and there exists a vector function $v\in W_{v^*}^{1,\infty}(\Omega_T;\R^n)$   such that  $v\equiv v^*$  in  $\Omega_T\setminus G$, $v\in C^1_{piece}(G;\R^n)$, $\dv v=u$ and  $(Du,v_t)\in \mathcal S_{dc}\cup \mathcal F$ a.e.\;in $\Omega_T$, and   $(Du,v_t)\in \mathcal S_{dc}$ a.e.\;in $\Omega_T^1$.
Since both $u$ and $v$ are piecewise $C^1$ in $G$,  there exists a sequence of disjoint open sets $\{G_j\}_{j=1}^\infty$ in $G$ with  $|\partial G_j|=0$ such that
\[
u\in C^1(\bar G_j),
\;v\in C^1(\bar G_j;\R^n)\; \; \forall j\geq 1,\quad |G\setminus \cup_{j=1}^\infty G_j|=0.
\]
We also choose an open set $G_0\subset\subset\Omega_T^1\setminus\bar G$ with $|\partial G_0|=0$ such that
\begin{equation}\label{density-1}
\left\{
\begin{split}
& \int_{(\Omega_T^1\setminus\bar G)\setminus G_0} |v_t-A(Du)|dz\le \frac{\epsilon}{5}|\Omega_T^1|,\\
& \int_{(\Omega_T^1\setminus\bar G)\setminus G_0} \mathrm{dist}((Du,v_t),\mathcal B)dz\le \frac{\epsilon}{5}|\Omega_T^1|.
\end{split}\right.
\end{equation}
Then the open set $\tilde G:=G_0\cup G$ is such that $\tilde G\subset\subset\Omega_T^1$, $|\partial\tilde G|=0$, and $\{G_j\}_{j=0}^\infty$ is a sequence of disjoint open subsets of $\tilde G$ whose union has measure $|\tilde G|$.

2. Let $j\in \{0\}\cup\N=:\N_0$ be fixed. Note that $(Du(z),v_t(z))\in\bar{\mathcal{S}}_{dc}$ for all $z=(x,t)\in G_j$ and that $H_j:=\{z\in G_j\;|\;(Du(z),v_t(z))\in\partial\mathcal{S}_{dc}\}$ is a (relatively) closed set in $G_j$ with measure zero. So $G_j':=G_j\setminus H_j$ is an open subset of $G_j$ with $| G_j'|=|G_j|$, and $(Du(z),v_t(z))\in\mathcal{S}_{dc}$ for all $z\in G_j'$. Now, we choose an open set $G_j''\subset\subset G_j'$ with $|\partial G_j''|=0$ such that
\begin{equation}\label{density-2}
\left\{
\begin{split}
& \int_{G_j'\setminus G_j''} |v_t-A(Du)|dz\le \frac{\epsilon}{5\cdot 2^{j+1}}|\Omega_T^1|,\\
& \int_{G_j'\setminus G_j''} \mathrm{dist}((Du,v_t),\mathcal B)dz\le \frac{\epsilon}{5\cdot 2^{j+1}}|\Omega_T^1|.
\end{split}\right.
\end{equation}
Observe that $(Du(z),v_t(z))\in\mathcal{S}_{dc}$ for all $z$ in the compact set $\bar G_j''$; so we are able to choose a finite index set $K_j\subset \N$ such that
\[
(Du(z),v_t(z))\in\bigcup_{k\in K_j}\mathcal{S}_{r_k-\bar\mu_{r_k},r_k+\bar\mu_{r_k}}\quad\forall z\in\bar G_j''.
\]
Let $k^j_1<\cdots<k^j_{n_j}$ denote the indices in $K_j$, where $n_j$ is the cardinality of $K_j$, and write
\[
\mathcal S^j_l=\mathcal{S}_{r_{k^j_l}-\bar\mu_{r_{k^j_l}},r_{k^j_l}+\bar\mu_{r_{k^j_l}}}\quad\mbox{for $l=1,\cdots,n_j$}.
\]

3. For each $l\in \{1,\cdots,n_j\}$ and each $\tau>0$, let
\[
\mathcal G^j_{l,\tau}=\left\{(p,\beta)\in  \mathcal S^j_l \;|\; \mathrm{dist}((p,\beta),\partial\mathcal S^j_l)> \tau,\,\mathrm{dist}((p,\beta),\mathcal B)>\tau \right\};
\]
then
\[
\mathcal S^j_l  =  (\cup_{\tau>0}\mathcal G^j_{l,\tau})\cup\{(p,\beta)\in\mathcal S^j_l\;|\;\mathrm{dist}((p,\beta),\mathcal B)=0\}.
\]
Note
\[
\begin{split}
\int_{G_j''} & |v_t (z)- A(Du(z))|\,dz\\
& \le \sum_{l=1}^{n_j}\int_{\{z\in G_j'' \,|\, (Du(z),v_t(z))\in \mathcal S^j_l\}}  |v_t (z)- A(Du(z))|\,dz\\
& = \sum_{l=1}^{n_j}\lim_{\tau\to 0^+}\int_{\{z\in G_j''\,|\,(Du(z),v_t(z))\in \mathcal G^j_{l,\tau}\}} |v_t(z)-A(Du(z))|\,dz,
\end{split}
\]
\[
\begin{split}
\int_{G_j''} & \mathrm{dist}((Du(z),v_t(z)),\mathcal B)\,dz\\
& \le \sum_{l=1}^{n_j}\lim_{\tau\to 0^+}\int_{\{z\in G_j''\,|\,(Du(z),v_t(z))\in \mathcal G^j_{l,\tau}\}}  \mathrm{dist}((Du(z),v_t(z)),\mathcal B)\,dz;
\end{split}
\]
then we choose a $\tau_j>0$ so that for $l=1,\cdots,n_j$, we have $|\partial O^j_{l}|=0$ and
\begin{equation}\label{density-3}
\left\{
\begin{split}
& \int_{F^j_{l}} |v_t(z)-A(Du(z))|\,dz < \frac{\epsilon}{5\cdot2^{j+1}n_j}|\Omega^1_T|,\\
& \int_{F^j_{l}} \mathrm{dist}((Du(z),v_t(z)),\mathcal B)\,dz < \frac{\epsilon}{5\cdot2^{j+1}n_j}|\Omega^1_T|,
\end{split}\right.
\end{equation}
where $O^j_{l}=  \{z\in G_j''\; |\; (Du(z),v_t(z))\in \mathcal G^j_{l,\tau_j}\}$ and $F^j_{l}=G_j''\setminus O^j_{l}$.  We also define $U^j_1=O^j_{1}$ and $U^j_l=O^j_l\setminus(\bar O^j_{1}\cup\cdots\cup\bar O^j_{l-1})$ $(l=2,\cdots,n_j)$; then $U^j_1,\cdots,U^j_{n_j}$ are disjoint open subsets of $\cup_{l=1}^{n_j}O^j_l$ whose union has measure $|\cup_{l=1}^{n_j}O^j_l|$.

4. We now fix an  $l\in \{1,\cdots,n_j\}$. Note that
\[
U^j_l\subset O^j_{l}=  \{z\in G_j''\; |\; (Du(z),v_t(z))\in \mathcal G^j_{l,\tau_j}\}
\]
and that  $\mathcal K^j_l:=\bar{\mathcal G}^j_{l,\tau_j}$ is a compact subset of $\mathcal S^j_l=\mathcal{S}_{r_{k^j_l}-\bar\mu_{r_{k^j_l}},r_{k^j_l}+\bar\mu_{r_{k^j_l}}}$. Let $\tilde Q\subset\R^n$ be a  box with $\Omega\subset\tilde Q$, and let $\tilde I=(0,T)$. Applying Theorem \ref{main-lemma} to box $\tilde Q\times\tilde I,$ $\mathcal K^j_l\subset\subset\mathcal S^j_l$, and $\epsilon'=\frac{\epsilon|\Omega_T|}{20}$, we obtain a constant $\delta^j_l>0$ that satisfies  the conclusion of the theorem. By the uniform continuity of $A$ on compact subsets of $\R^n$, we can find a $\theta=\theta_{\epsilon,s_+}>0$ such that
\begin{equation}\label{density-4}
|A(p)-A(p')|<\frac{\epsilon}{20}
\end{equation}
whenever $|p|,\,|p'|\le s_+$ and $|p-p'|\le \theta$. %, where the number $\bar s>0$ will be chosen later.
 Also by the uniform continuity of $u$, $v$ and their gradients on $\bar U^j_l$,
there exists a $\nu^j_l>0$ such that
\begin{equation}\label{density-5}
\begin{array}{c}
  |u(z)-u(z')|+|\nabla u(z)-\nabla u(z')| +|v(z)-v(z')| \\
  +|\nabla v(z)-\nabla v(z')| <\min\{\frac{\delta^j_{l}}{2},\frac{\epsilon}{20},\theta\} \end{array}
\end{equation}
whenever $z,z'\in \bar U^j_l$ and  $|z-z'|\le \nu^j_l.$ We now cover $U^j_l$ (up to measure zero) by a sequence of disjoint boxes $\{Q^j_{l,i}\times I^j_{l,i}\}_{i=1}^\infty$ in $U^j_l$ with center $z^j_{l,i}$ and diameter $d^j_{l,i}<\nu^j_l.$

5. Fix an $i\in\mathbb{N}$, and set  $w=(u,v)$ and  $\xi=\begin{pmatrix} p&c\\B & \beta\end{pmatrix}=\nabla w(z^j_{l,i})=\begin{pmatrix} Du(z^j_{l,i}) & u_t(z^j_{l,i})\\Dv(z^j_{l,i}) & v_t(z^j_{l,i})\end{pmatrix}.$   By the choice of $\delta^j_l>0$ in Step 4 via Theorem \ref{main-lemma}, since $Q^j_{l,i}\times I^j_{l,i}\subset\tilde Q\times\tilde I$ and $(p,\beta)\in\mathcal K^j_l$,  for all sufficiently small $\rho>0$, there exists a function $\omega^j_{l,i}=(\varphi^j_{l,i},\psi^j_{l,i})\in C^\infty_c(Q^j_{l,i}\times I^j_{l,i};\R^{1+n})$ satisfying

(a) \; $\dv \psi^j_{l,i} =0$  in $Q^j_{l,i}\times I^j_{l,i}$,

(b) \; $(p'+D\varphi^j_{l,i}(z), \beta'+ (\psi^j_{l,i})_t(z))\in \mathcal{S}^j_l$ for all $z\in Q^j_{l,i}\times I^j_{l,i}$

\hspace{8mm} and all $|(p',\beta')-(p,\beta)|\leq\delta^j_l,$

(c) \; $\|\omega^j_{l,i}\|_{L^\infty(Q^j_{l,i}\times I^j_{l,i})}<\rho,$

(d) \; $\int_{Q^j_{l,i}\times I^j_{l,i}} |\beta+(\psi^j_{l,i})_t(z)-A(p+D\varphi^j_{l,i}(z))|dz<\epsilon' |Q^j_{l,i}\times I^j_{l,i}|/|\tilde Q\times \tilde I|,$

(e) \; $\int_{Q^j_{l,i}\times I^j_{l,i}} \mathrm{dist}((p+D\varphi^j_{l,i}(z),\beta+(\psi^j_{l,i})_t(z)),\mathcal A^j_l)dz  <\epsilon' |Q^j_{l,i}\times I^j_{l,i}|/|\tilde Q\times \tilde I|,$

(f) \; $\int_{Q^j_{l,i}} \varphi^j_{l,i}(x,t)dx=0$  for all $t\in I^j_{l,i},$

(g) \; $\|(\varphi^j_{l,i})_t\|_{L^\infty(Q^j_{l,i}\times I^j_{l,i})}<\rho,$\\
where the set $\mathcal A^j_l=\mathcal A_{r_{k^j_l}-\bar\mu_{r_{k^j_l}},r_{k^j_l}+\bar\mu_{r_{k^j_l}}}\subset \R^{n+n}$ is as in Theorem \ref{main-lemma}.
Here, we also let $0<\rho\leq\min\{\bar\tau_0,\frac{\delta^j_l}{2C},\frac{\epsilon}{20C},\eta\}$, where $C_n>0$ is the constant in Theorem \ref{div-inv} and $C$ is the product of $C_n$ and the sum of the lengths of all sides of $\tilde Q$.
From $\varphi^j_{l,i}|_{\partial(Q^j_{l,i}\times I^j_{l,i})}\equiv 0$ and (f), we can apply Theorem \ref{div-inv} to $\varphi^j_{l,i}$ on $Q^j_{l,i}\times I^j_{l,i}$ to obtain a function $g^j_{l,i}=\mathcal R\varphi^j_{l,i}\in C^1( \overline{Q^j_{l,i}\times I^j_{l,i}};\R^n)\cap W^{1,\infty}_0( Q^j_{l,i}\times I^j_{l,i};\R^n)$
such that $\dv g^j_{l,i}=\varphi^j_{l,i}$ in $Q^j_{l,i}\times I^j_{l,i}$ and
\begin{equation}\label{density-6}
\|(g^j_{l,i})_t\|_{L^\infty(Q^j_{l,i}\times I^j_{l,i})}\leq C\|(\varphi^j_{l,i})_t\|_{L^\infty(Q^j_{l,i}\times I^j_{l,i})}\leq\frac{\delta^j_l}{2}.\;\;\mbox{(by (g))}
\end{equation}

6. As $|v_t-A(Du)|,\,\mathrm{dist}((Du,v_t),\mathcal B)\in L^\infty(\Omega_T)$, we can select a finite index set $\mathcal I\subset \{(j,l)\,|\,j\in\N_0,\,1\le l\le n_j\}\times\N=:\mathcal{J}$ such that
\begin{equation}\label{density-7}
\left\{
\begin{split}
& \int_{\bigcup_{(j,l,i)\in \mathcal{J}\setminus\mathcal{I}}Q^j_{l,i}\times I^j_{l,i}}|v_t(z)-A(Du(z))|dz\le\frac{\epsilon}{5}|\Omega_T^1|,\\
& \int_{\bigcup_{(j,l,i)\in \mathcal{J}\setminus\mathcal{I}}Q^j_{l,i}\times I^j_{l,i}} \mathrm{dist}((Du(z),v_t(z)),\mathcal B) dz\le\frac{\epsilon}{5}|\Omega_T^1|.
\end{split}\right.
\end{equation}
We finally define
\[
(\tilde u,\tilde v)=(u,v)+\sum_{(j,l,i)\in\mathcal I}\chi_{Q^j_{l,i}\times I^j_{l,i}}(\varphi^j_{l,i},\psi^j_{l,i}+g^j_{l,i})\;\;\textrm{ in $\Omega_T$.}
\]
%As a side remark, note here that only \emph{finitely} many functions $(\varphi^i_j,\psi^i_j+g^i_j) $ are disjointly patched  to the original $(u,v)$ to obtain a new function $(\tilde u,\tilde v)$ towards the goal of the proof. The reason for using only finitely many pieces of gluing is due to the lack of control over the spatial gradients $D(\psi^i_j+g^i_j)$, and overcoming this difficulty is at the heart of  this paper.

7. Let us finally check that $\tilde u$  together with  $\tilde v$ indeed gives the desired result. By construction, it is clear that $\tilde G\subset\subset\Omega_T^1$, $|\partial\tilde G|=0$, $\tilde u=u=u^*$ and $\tilde v=v=v^*$ in $\Omega_T\setminus\tilde G$, $\tilde u\in C^1_{piece}(\tilde G)\cap W^{1,\infty}_{u^*}(\Omega_T)$, and $\tilde v\in C^1_{piece}(\tilde G;\R^n)\cap W^{1,\infty}_{v^*}(\Omega_T;\R^n)$.
By the choice of $\rho$ in (g) as $\rho\le\bar\tau_0$, we have $\|\tilde u_t\|_{L^\infty(\Omega_T)}<m.$
Next, let $(j,l,i)\in\mathcal I,$ and observe that for $z\in Q^j_{l,i}\times I^j_{l,i}$, with $(p,\beta)=(Du(z^j_{l,i}),v_t(z^j_{l,i}))\in\mathcal G^j_{l,\tau_j}$, since $|z-z^j_{l,i}|<d^j_{l,i}<\nu^j_i$, it follows from (\ref{density-5}) and (\ref{density-6}) that
\[
|(Du(z),v_t(z)+(g^j_{l,i})_t(z))-(p,\beta)|\leq\delta^j_l,
\]
and so $(D\tilde u(z),\tilde v_t(z))\in\mathcal S^j_l\subset\mathcal S_{dc}$ from (b) above. From (a) and $\dv g^j_{l,i}=\varphi^j_{l,i}$, for $z\in Q^j_{l,i}\times I^j_{l,i}$,
\[
\dv \tilde v(z)=\dv (v+\psi^j_{l,i}+g^j_{l,i})(z)=u(z)+0+\varphi^j_{l,i}(z)=\tilde u(z).
\]
Therefore, $\tilde u\in\mathcal U$.
Next, observe
\[
\int_{\Omega_T}|\tilde v_t-A(D\tilde u)|dz = \int_{\Omega_T\setminus {\tilde G}}| v_t^*-A(Du^*)|dz+\int_{\tilde G}|\tilde v_t-A(D\tilde u)|dz
\]
\[
= \int_{\Omega_T^0\cup\Omega_T^2\cup\Omega_T^3}| v_t^*-\tilde A(Du^*)|dz+\int_{\Omega_T^1\setminus {\tilde G}}| v_t^*-A(Du^*)|dz+ \int_{\tilde G}|\tilde v_t-A(D\tilde u)|dz
\]
\[
\begin{split}
\le & \int_{\Omega_T^1\setminus {\tilde G}}| v_t-A(Du)|dz+\sum_{j=0}^\infty\int_{G_j'\setminus G_j''}|v_t-A(Du)|dz  \\
& +\sum_{j=0}^\infty\sum_{l=1}^{n_j}\int_{F^j_l}|v_t-A(Du)|dz +\sum_{(j,l,i)\in\mathcal J\setminus\mathcal I}\int_{Q^j_{l,i}\times I^j_{l,i}}|v_t-A(D u)|dz \\
& +\sum_{(j,l,i)\in \mathcal I}\int_{Q^j_{l,i}\times I^j_{l,i}}|\tilde v_t-A(D\tilde u)|dz \\
=: & I^1_1+I^1_2+I^1_3+I^1_4+I^1_5,
\end{split}
\]
\[
\int_{\Omega_T}\mathrm{dist}((D\tilde u,\tilde v_t),\mathcal B)dz = \int_{\Omega_T\setminus {\tilde G}}\mathrm{dist}((Du^*, v_t^*),\mathcal B)dz+\int_{\tilde G}\mathrm{dist}((D\tilde u,\tilde v_t),\mathcal B)dz
\]
\[
\begin{split}
\le & \int_{\Omega_T^1\setminus {\tilde G}}\mathrm{dist}((D u, v_t),\mathcal B)dz+\sum_{j=0}^\infty\int_{G_j'\setminus G_j''}\mathrm{dist}((D u, v_t),\mathcal B)dz  \\
& +\sum_{j=0}^\infty\sum_{l=1}^{n_j}\int_{F^j_l}\mathrm{dist}((D u, v_t),\mathcal B)dz +\sum_{(j,l,i)\in\mathcal J\setminus\mathcal I}\int_{Q^j_{l,i}\times I^j_{l,i}}\mathrm{dist}((D u, v_t),\mathcal B)dz \\
& +\sum_{(j,l,i)\in \mathcal I}\int_{Q^j_{l,i}\times I^j_{l,i}}\mathrm{dist}((D\tilde u,\tilde v_t),\mathcal B)dz \\
=: & I^2_1+I^2_2+I^2_3+I^2_4+I^2_5.
\end{split}
\]
From (\ref{density-1}), (\ref{density-2}), (\ref{density-3}) and (\ref{density-7}), we have $I^k_1+I^k_2+I^k_3+I^k_4\leq\frac{4\epsilon}{5}|\Omega_T^1|\;(k=1,2)$. Note   that for $(j,l,i)\in\mathcal I$ and $z\in Q^j_{l,i}\times I^j_{l,i},$ from (\ref{density-1}), (\ref{density-2}) and (g),
\[
|\tilde v_t(z)-A(D\tilde u(z))|=|v_t(z)+(\psi^j_{l,i})_t(z)+(g^j_{l,i})_t(z)-A(Du(z)+D\varphi^j_{l,i}(z))|
\]
\[
\leq |v_t(z)-v_t(z^j_{l,i})|+|v_t(z^j_{l,i})+(\psi^j_{l,i})_t(z)-A(Du(z^j_{l,i})+D\varphi^j_{l,i}(z))|
\]
\[
+|(g^j_{l,i})_t(z)|+|A(Du(z^j_{l,i})+D\varphi^j_{l,i}(z))-A(Du(z)+D\varphi^j_{l,i}(z))|
\]
\[
\le\frac{\epsilon}{10}+|v_t(z^j_{l,i})+(\psi^j_{l,i})_t(z)-A(Du(z^j_{l,i})+D\varphi^j_{l,i}(z))|
\]
\[
+|A(Du(z^j_{l,i})+D\varphi^j_{l,i}(z))-A(Du(z)+D\varphi^j_{l,i}(z))|.
\]
Similarly, since $\mathcal A^j_l\subset\mathcal B$, we have
\[
\begin{split}
\mathrm{dist}( & (D\tilde u(z),\tilde v_t(z)),\mathcal B)\\
& \le \frac{\epsilon}{10}+\mathrm{dist}((Du(z^j_{l,i})+D\varphi^j_{l,i}(z),v_t(z^j_{l,i})+(\psi^j_{l,i})_t(z)),\mathcal B)\\
& \le \frac{\epsilon}{10}+\mathrm{dist}((Du(z^j_{l,i})+D\varphi^j_{l,i}(z),v_t(z^j_{l,i})+(\psi^j_{l,i})_t(z)),\mathcal A^j_l).
\end{split}
\]
From (b) and (i) of Theorem \ref{thm:main2}, we have
$|Du(z^j_{l,i})+D\varphi^j_{l,i}(z)|\le s_+(r_{k^j_l}+\bar\mu_{r_{k^j_l}})<s_+$.
As $(D\tilde u(z),\tilde v_t(z))\in\mathcal S^j_l$, we also have $|Du(z)+D\varphi^j_{l,i}(z)|=|D\tilde u(z)|< s_+$, and by (\ref{density-5}), $|Du(z^i_j)-Du(z)|<\theta$. From (\ref{density-4}), we thus have
\[
|A(Du(z^j_{l,i})+D\varphi^j_{l,i}(z))-A(Du(z)+D\varphi^j_{l,i}(z))|<\frac{\epsilon}{20}.
\]
Integrating the two inequalities above over $Q^i_j\times I^i_j,$ we now obtain from (d) and (e), respectively, that
\[
\int_{Q^j_{l,i}\times I^j_{l,i}}|\tilde v_t(z)-A(D\tilde u(z))|dz\le\frac{3\epsilon}{20}|Q^j_{l,i}\times I^j_{l,i}|+\frac{\epsilon|\Omega_T|}{20}\frac{|Q^j_{l,i}\times I^j_{l,i}|}{|\tilde Q\times \tilde I|}\leq\frac{\epsilon}{5}|Q^j_{l,i}\times I^j_{l,i}|,
\]
\[
\int_{Q^j_{l,i}\times I^j_{l,i}} \mathrm{dist}( (D\tilde u(z),\tilde v_t(z)),\mathcal B) dz\le\frac{\epsilon}{10}|Q^j_{l,i}\times I^j_{l,i}|+\frac{\epsilon|\Omega_T|}{20}\frac{|Q^j_{l,i}\times I^j_{l,i}|}{|\tilde Q\times \tilde I|}\leq\frac{\epsilon}{5}|Q^j_{l,i}\times I^j_{l,i}|;
\]
thus $I^k_5\leq\frac{\epsilon}{5}|\Omega_T^1|$, and so $I^k_1+I^k_2+I^k_3+I^k_4+I^k_5\le\epsilon|\Omega_T^1|$, where $k=1,2$. Therefore, $\tilde u\in\mathcal U_\epsilon.$
Lastly, from (c) with $\rho\le \eta$ and the definition of $\tilde u$, we have $\|\tilde u-u\|_{L^\infty(\Omega_T)}<\eta$.

The proof is now complete.
\end{proof}

\subsection{Completion of the proof of Theorem \ref{thm:NF-1}}

Unless specifically distinguished, the proof below is common for both \textbf{Type I: FFT} solutions and \textbf{Type II: BFT} solutions.

\begin{proof}[Proof of Theorem \ref{thm:NF-1}]
We return to Section \ref{sec:add-set}. As outlined in Remark \ref{rmk-1}, Theorem \ref{thm-density-1}  and Theorem \ref{thm:main} together give infinitely many Lipschitz solutions $u$ to problem (\ref{ib-P}).

We now follow the proof of Theorem \ref{thm:main} for detailed information on such a Lipschitz solution $u\in\mathcal G$ to (\ref{ib-P}). Here $Du$ is the a.e.-pointwise limit of some sequence $Du_j$, where the sequence $u_j\in\mathcal U_{1/j}$ converges to $u$ in $L^\infty(\Omega_T)$. Since $u_j\equiv u^*$ in $\Omega_T\setminus G_j$ for some open set $G_j\subset\subset\Omega_T^1$, we also have $u\equiv u^*\in C^{2+\alpha,1+\alpha/2}(\bar\Omega_T^{\tilde r})$  so that
\[
u_t=\dv(A(Du))\;\;\mbox{and}\;\; |Du|>s_+(\tilde r)\;\;\mbox{in}\;\;\Omega_T^{\tilde r}=\Omega_T^3
\]
and
\begin{equation}\label{NF-1-1}
|Du|=0\;\;\mbox{a.e. in}\;\;\Omega_T^0,\;\;\mbox{and}\;\;|Du|=s_+(\tilde r)\;\;\mbox{a.e. in}\;\;\Omega_T^2.
\end{equation}
Note $(v_j)_t \wcon v_t$\, in\, $L^2(\Omega_T;\R^n)$, where $v_j$ is the corresponding vector function to $u_j$ and $v\in W^{1,2}((0,T);L^2(\Omega;\R^n))$. From (\ref{div-v3}), we can even deduce that  $(v_j)_t \to v_t$\, pointwise a.e. in $\Omega_T$. On the other hand, from the definition of $\mathcal U_{1/j}$,
\[
\int_{\Omega_T^1} \mathrm{dist}((Du_j,(v_j)_t),\mathcal B)\, dxdt\le\frac{1}{j}|\Omega_T^1|\to 0\;\;\mbox{as}\;\; j\to \infty;
\]
thus $(Du,v_t)\in\mathcal B$\, a.e. in $\Omega_T^1$, yielding together with (\ref{NF-1-1}) that
\[
|Du|\in[s_-^2(\tilde r),s_+(\tilde r)]\cup\{0\}\;\;\mbox{in}\;\;\Omega_T\setminus\Omega_T^{\tilde r},\quad (\textbf{Type I})
\]
\[
|Du|\in[0,s_-^1(\tilde r)]\cup[s_0,s_+(\tilde r)]\;\;\mbox{in}\;\;\Omega_T\setminus\Omega_T^{\tilde r},\quad (\textbf{Type II})
\]

The proof is now complete.
\end{proof}

\subsection{Proof of Theorem \ref{thm:NF-2}}
Let $u_0\in C^{2+\alpha}(\bar\Omega)$ with $Du_0\cdot\n|_{\partial\Omega}=0$.
If $\|Du_0\|_{L^\infty(\Omega)}=0$, that is, $u_0\equiv c$ in $\Omega$ for some constant $c\in\R$, the constant function $u\equiv c$ in $\Omega_T$ is a Lipschitz solution to problem (\ref{ib-P}).
The existence of infinitely many Lipschitz solutions to (\ref{ib-P}) when $|Du_0(x_0)|\in (0,s_+)$ for some $x_0\in\Omega$ is simply the result of Theorem \ref{thm:NF-1}. So we cover the remaining case here.

Assume $\min_{\bar\Omega}|Du_0|\ge s_+$. Fix any number $0<r=\tilde r<\sigma(s_+)$, and let $\tilde\sigma,\,\tilde f\in C^{1+\alpha}([0,\infty))$ be some functions from Lemma \ref{lem:modi-NF}. Using the flux $\tilde A(p)=\tilde f(|p|^2)p$, Theorem \ref{existence-gr-max} gives a unique solution $u^*\in C^{2+\alpha,1+\alpha/2}(\bar\Omega_T)$ to problem (\ref{ib-parabolic}). If $|Du^*|$ stays on or above the threshold $s_+$ in $\Omega_T$, then $u^*$ itself is a Lipschitz solution to (\ref{ib-P}). Otherwise, set $\bar s=\frac{s_0+s_+}{2}$ and choose a point $(\bar x,\bar t)\in\Omega_T$ such that
\[
|Du^*|\ge\bar s\;\;\mbox{in}\;\;\Omega\times(0,\bar t),\quad |Du^*(\bar x,\bar t)|\in (0,s_+).
\]
Regarding $u_1(\cdot):=u^*(\cdot,\bar t)\in C^{2+\alpha}(\bar\Omega)$, satisfying $Du_1\cdot\n|_{\partial\Omega}=0$, as a new initial datum at time $t=\bar t$, it follows from Theorem \ref{thm:NF-1} that problem (\ref{ib-P}), with the initial datum $u_1$ at  time $t=\bar t$, admits infinitely many Lipschitz solutions $\bar u$ in $\Omega\times(\bar t,T)$. Then the patched functions $u=\chi_{\Omega\times (0,\bar t)} u^*+\chi_{\Omega\times [\bar t,T)} \bar u$ in $\Omega_T$ become Lipschitz solutions to the original problem (\ref{ib-P}), and the proof is complete.

\section{Further remarks}\label{sec:remarks}

In this final section, we briefly give an overview of how one can combine \cite{KY2} with this paper to deduce further existence results. Instead of trying to formulate a certain general result, we present some case-by-case existence results in a casual manner. Also, even if one can possibly study the Dirichlet- and mixed-boundary value problems of forward-backward parabolic equations by using the methods of the two papers, we focus only on the Neumann problem (\ref{ib-P}) with mixed type (i.e., mixture of Perona-Malik, H\"ollig and non-Fourier types)  profiles $\sigma=\sigma(s):[0,\infty)\to\R$ for diffusion fluxes $A(p)=\frac{\sigma(|p|)}{|p|}p.$  Throughout this section, we assume that the domain  $\Omega$ and initial datum $u_0$ satisfy condition (\ref{assume-1}). In addition, all the profiles $\sigma(s)$ considered here are assumed to have derivative values lying in some interval $[\lambda,\Lambda]$ $(\Lambda\ge\lambda>0)$ for all sufficiently large $s>0$.

As an example, consider the profile $\sigma(s)$ given by the following graph.

\begin{figure}[ht]
\begin{center}
\begin{tikzpicture}[scale =1]
    \draw[->] (-.5,0) -- (11.3,0);
	\draw[->] (0,-.5) -- (0,5.2);
 \draw[dashed] (0.2,0.5)--(3,0.5);
 \draw[dashed] (0.9,1.53)--(4.45,1.53);
 \draw[dashed] (5.18,2.48)--(7.6,2.48);
 \draw[dashed] (6,3.5)--(9.7,3.5);
   \draw[dashed] (0.2,0.5)--(0.2,0);
   \draw[dashed] (4.45, 1.53)--(4.45,0) ;
   \draw[dashed] (5.18,2.48)--(5.18,0) ;
   \draw[dashed] (9.7,3.5)  --  (9.7,0) ;
     \draw[thick] (0.2,0)--(4.45,0);
     \draw[thick] (5.18,0)--(9.7,0);
	\draw[thick]   (0, 0) .. controls (1,3) and  (1.2, 0.7)   ..(3,0.5);
	\draw[thick]   (3, 0.5) .. controls  (4, 0.6) and (5,2.5) ..(6, 3.5 );
    \draw[thick]   (6, 3.5) .. controls  (7, 2.8) and (8,1) ..(11, 5 );
%\draw[thick]   (0.9,1.5) .. controls  (1,1.6) and (2,2.6) ..(11,2.8);
	\draw (11.3,0) node[below] {$s$};
    \draw (-0.3,0) node[below] {{$0$}};
    \draw (10.5, 3.5) node[above] {$\sigma(s)$};
 %\draw (0, 3.3) node[left] {$\sigma(s_1)$};
 %\draw (0, 1.7) node[left] {$\sigma(s_2)$};
% \draw[dashed] (0,1.5)--(9.2, 1.5);
% \draw[dashed] (0,3)--(7.4, 3);
% \draw[dashed] (7.4,0)--(7.4, 3);
%\draw[dashed] (9.2,0)--(9.2, 1.5);
   \draw (0.2, 0) node[below] {$s_1$};
   \draw (4.45, 0) node[below] {$s_2$};
   \draw (5.18, 0) node[below] {$s_3$};
   \draw (9.7, 0) node[below] {$s_4$};
% \draw[dashed] (0.9,0)--(0.9, 1.5);
    \end{tikzpicture}
\end{center}
\caption{The first example of profile  $\sigma(s).$}
\label{fig5}
\end{figure}
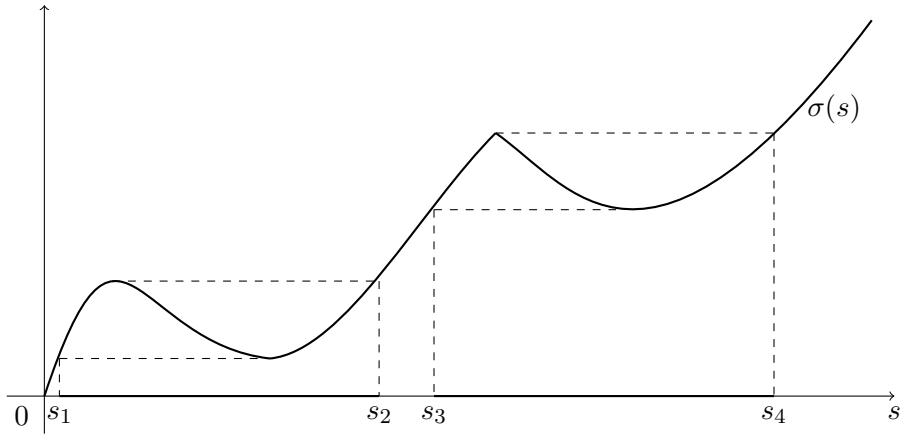

For such a profile $\sigma(s)$, if the initial gradient size $|Du_0|$ belongs to the \emph{phase transition zones} $(s_1,s_2)$ and $(s_3,s_4)$ at some points in $\Omega$, one can employ the method of \cite{KY2} to generate solutions. In doing so, one may choose a modified profile $\tilde \sigma(s)$ that is equal to the original
$\sigma(s)$ outside $(s_1,s_2)\cup(s_3,s_4)$ and whose derivative values always belong to some interval $[\theta,\Theta]$ $(\Theta\ge\theta>0)$. Then  obtained solutions will  be smooth evolutions in the subdomain of $\Omega_T$ in which a certain classical solution $u^*$ corresponding to the profile $\tilde\sigma(s)$ with initial datum $u_0$ has  gradient size lying outside $(s_1,s_2)\cup(s_3,s_4)$. In the subdomain of $\Omega_T$ in which $|Du^*|$ lies in two fixed disjoint open intervals in $(s_1,s_2)\cup(s_3,s_4)$, where we do \emph{laminate} (i.e., convexify in the rank-one sense) certain matrix sets to capture some open structures that enable us to do a \emph{surgery}, such solutions will be highly oscillatory as those should have gradient size that  belongs only to four disjoint intervals. The only difference from \cite{KY2} in getting solutions is that we formulate two matrix sets to be laminated instead of one. Obviously, finitely many ups and downs of the graph of a profile $\sigma(s)$ can be also dealt in the same way for existence as long as the graph is  increasing like a staircase as in Figure \ref{fig5}. However, profiles of non-staircase shape are still possible to handle  as we explain below.

We next consider the profile $\sigma(s)=\sin s$ whose graph is given below. Assume in this case that the space domain $\Omega$ is convex to guarantee the gradient maximum principle for uniformly parabolic diffusions \cite[Theorem 2.1]{KY1}.

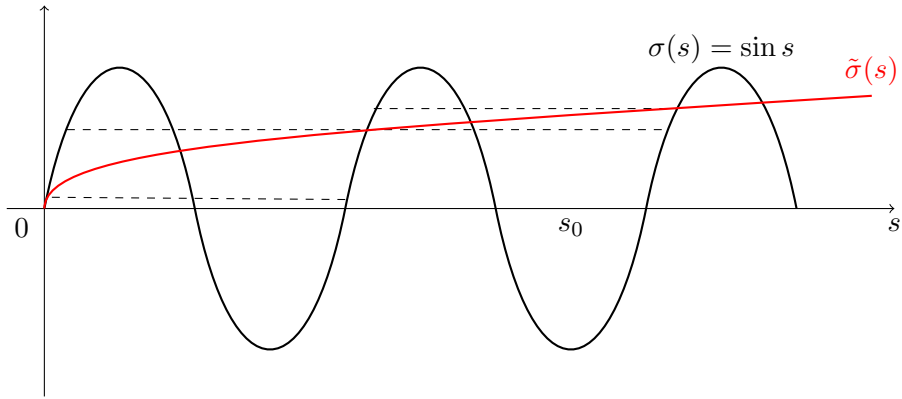
\begin{figure}[ht]
\begin{center}
\begin{tikzpicture}[scale =1]
    \draw[->] (-.5,2.5) -- (11.3,2.5);
	\draw[->] (0,0) -- (0,5.2);
 \draw[dashed] (0.1,2.65)--(4,2.62);
 %\draw[dashed] (0.2,3.25)--(4.2,3.25);
 \draw[dashed] (0.3,3.55)--(8.31,3.55);
 %\draw[dashed] (4.32,3.65)--(8.33,3.65);
 \draw[dashed] (4.38,3.83)--(8.4,3.83);
   %\draw[dashed] (0.2,0.5)--(0.2,0);
   %\draw[dashed] (4.45, 1.53)--(4.45,0) ;
   %\draw[dashed] (5.18,2.48)--(5.18,0) ;
   %\draw[dashed] (9.7,3.5)  --  (9.7,0) ;
     %\draw[thick] (0.2,0)--(4.45,0);
     %\draw[thick] (5.18,0)--(9.7,0);
	\draw[thick]   (0, 2.5) .. controls (0.5,5) and  (1.5, 5)   ..(2,2.5);
	\draw[thick]   (2, 2.5) .. controls (2.5,0) and  (3.5, 0)   ..(4,2.5);
    \draw[thick]   (4, 2.5) .. controls (4.5,5) and  (5.5, 5)   ..(6,2.5);
	\draw[thick]   (6, 2.5) .. controls (6.5,0) and  (7.5, 0)   ..(8,2.5);
    \draw[thick]   (8, 2.5) .. controls (8.5,5) and  (9.5, 5)   ..(10,2.5);
     \draw[thick,red]   (0,2.5) .. controls  (0.01,3.3) and (3,3.5) ..(11,4);
	\draw (11.3,2.5) node[below] {$s$};
    \draw (7,2.5) node[below] {$s_0$};
    \draw (-0.3,2.5) node[below] {{$0$}};
    \draw (9, 4.3) node[above] {$\sigma(s)=\sin s$};
    \draw (11, 4) node[above] {\textcolor{red}{$\tilde \sigma(s)$}};
 %\draw (0, 3.3) node[left] {$\sigma(s_1)$};
 %\draw (0, 1.7) node[left] {$\sigma(s_2)$};
% \draw[dashed] (0,1.5)--(9.2, 1.5);
% \draw[dashed] (0,3)--(7.4, 3);
% \draw[dashed] (7.4,0)--(7.4, 3);
%\draw[dashed] (9.2,0)--(9.2, 1.5);
   %\draw (0.2, 0) node[below] {$s_1$};
   %\draw (4.45, 0) node[below] {$s_2$};
   %\draw (5.18, 0) node[below] {$s_3$};
   %\draw (9.7, 0) node[below] {$s_4$};
% \draw[dashed] (0.9,0)--(0.9, 1.5);
    \end{tikzpicture}
\end{center}
\caption{The second example of profile  $\sigma(s).$}
\label{fig6}
\end{figure}

For instance, suppose $3\pi<s_0:=\sup_{\Omega}|Du_0|<4\pi$; then select a suitable modified profile $\tilde\sigma(s)$ as in Figure \ref{fig6}. Following \cite{KY2}, we can perform (simultaneous) lamination for two matrix sets to obtain solutions whose gradient size belongs to several disjoint intervals. Here, the gradient maximum principle for the classical solution $u^*$ corresponding to $\tilde \sigma(s)$ with initial datum $u_0$ is required to construct a admissible class containing a pair of functions with $u^*$ as its first component. If $|Du^*|$ were escaping beyond $s_0$, we would not be able to define the admissible class to problem (\ref{ib-P}) for any given $T>0$. In case of profile $\sigma(s)=-\sin s$, non-Fourier type diffusion may arise in the small gradient regime where $|Du^*|<3\pi/2$; thus the method of this paper can be combined with that of \cite{KY2} to obtain solutions for such a profile $\sigma(s)$.

Lastly, consider the profile $\sigma(s)$ having the graph as below. In this case, non-Fourier type diffusion may occur in the interval $(0,s_+)$, and H\"ollig type phase transitions arise in the interval $(s_1,s_2)$. Clearly, the former type can be handled by the technique of this paper, and the latter by that of \cite{KY2}. One essential difference between these two techniques is that although the one in the current paper stems from the other in \cite{KY2}, it makes use of countably many partial rank-one structures of the related matrix set to avoid a certain \emph{degeneracy} which occurs at the points $x\in\Omega$ for which $\sigma(|Du_0(x)|)=0$ when $t=0$. (Such a degeneracy does not appear in space dimension $n=1$.) Accordingly, compared to a single partial rank-one structure used in \cite{KY2}, more sophisticated  scheme has been used to detect partial rank-one structures and to combine those in this paper. %Towards the existence with the above profile $\sigma(s)$, one can naturally use both methods of \cite{KY2} and this paper.

\begin{figure}[ht]
\begin{center}
\begin{tikzpicture}[scale =1]
    %\draw[->] (-.5,0) -- (11.3,0);
	\draw[->] (0,0.5) -- (0,6.5);
    \draw[->] (-.5,2) -- (9,2);
  \draw[dashed] (0,0.95)--(1.65,0.95);
  \draw[dashed] (0,3)--(4.05,3);
  \draw[dashed] (4.65,3.85)--(6.9,3.85);
  \draw[dashed] (5.5,4.98)--(7.8,4.98);
    \draw[dashed] (1.65, 0.95)  --  (1.65, 2) ;
    \draw[dashed] (4.05, 3)  --  (4.05, 2) ;
    \draw[dashed] (4.65, 3.85)  --  (4.65, 2) ;
    \draw[dashed] (7.8, 4.98)  --  (7.8, 2) ;
    %\draw[dashed] (4.6, 1.9)  --  (4.6, 2.6) ;
      \draw[thick] (0,2) -- (4.05,2);
      \draw[thick] (4.65,2) -- (7.8,2);
	\draw[thick]   (0, 2) .. controls (1.5,0) and  (3, 1)   ..(5,4.5);
	\draw[thick]   (5, 4.5) .. controls  (6, 6.5) and (7,1) ..(8, 6 );
%\draw[thick]   (0.9,1.5) .. controls  (1,1.6) and (2,2.6) ..(11,2.8);
	\draw (9,2) node[below] {$s$};
    \draw (-0.3,2) node[below] {{$0$}};
    %\draw (-0.3,3.5) node[below] {$r$};
    %\draw (-0.4,2.18) node[below] {$-r$};
    \draw (8.5, 5) node[above] {$\sigma(s)$};
 \draw (0, 3) node[left] {$\sigma(s_+)$};
 \draw (0, 0.95) node[left] {$\sigma(s_-)=-\sigma(s_+)$};
% \draw[dashed] (0,1.5)--(9.2, 1.5);
% \draw[dashed] (0,3)--(7.4, 3);
% \draw[dashed] (7.4,0)--(7.4, 3);
%\draw[dashed] (9.2,0)--(9.2, 1.5);
   \draw (1.65, 2) node[above] {$s_-$};
   \draw (4.05, 2) node[below] {$s_+$};
   \draw (4.65, 2) node[below] {$s_1$};
   \draw (7.8, 2) node[below] {$s_2$};
% \draw[dashed] (0.9,0)--(0.9, 1.5);
    \end{tikzpicture}
\end{center}
\caption{The third example of profile  $\sigma(s).$}
\label{fig7}
\end{figure}
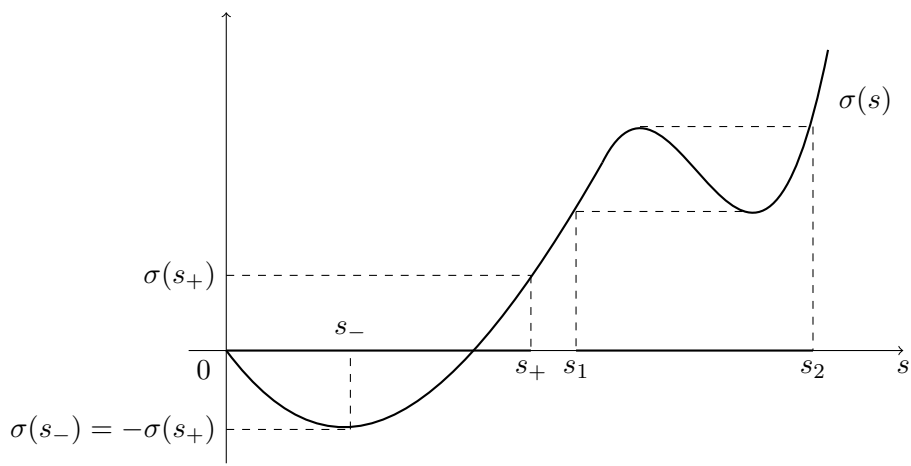

In concluding this paper, we remark that for many other possible profiles $\sigma(s)$ that are a combination of Perona-Malik, H\"ollig and non-Fourier types, one can expect the existence of solutions in \emph{almost} all cases. %except that an extra care is needed for Perona-Malik type for global solutions that requires a gradient maximum principle for classical solution to a modified standard problem, which may not be valid for non-convex domain $\Omega$ (see \cite{KY1}).

%\subsection{Other possible form of profiles for existence}

%\subsection{Dirichlet boundary value problem}

%\subsection{Co-existence of radial and non-radial solutions}


\begin{thebibliography}{11}

\bibitem{AR}
N. Alikakos and R. Rostamian, {\em Gradient estimates for degenerate diffusion equations. I}, Math. Ann., {\bf 259} (1) (1982), 53--70.


%\bibitem{BJ}
%J.M. Ball and R.D. James,  {\em Fine phase mixtures as minimizers of energy},  {Arch. Rational Mech. Anal.}, {\bf 100} (1) (1987), 13--52.


%\bibitem{BF}
%G. Bellettini and G. Fusco, {\em The $\Gamma$-limit and the related gradient flow for singular perturbation functionals of Perona-Malik type},  {Trans. Amer. Math. Soc.}, {\bf 360} (9) (2008),  4929--4987.

%\bibitem{BFG}
%G. Bellettini, G. Fusco and N. Guglielmi, {\em A concept of solution and numerical experiments for forward-backward diffusion equations},  {Discrete Contin. Dyn. Syst.}, {\bf 16} (4) (2006),  783--842.

\bibitem{BB}
J. Bourgain and H. Brezis, {\em On the equation $\dv Y=f$ and application to control of phases}, {J. Amer. Math. Soc.}, {\bf 16} (2) (2002), 393--426.

\bibitem{Br}
H. Br\'ezis, ``Op\'erateurs maximaux monotones et semi-groupes de contractions dans les espaces de Hilbert," North-Holland Mathematics Studies, No. 5. Notas de Matem�tica (50). North-Holland Publishing Co., Amsterdam-London; American Elsevier Publishing Co., Inc., New York, 1973.



\bibitem{BBT}
A. Bruckner, J. Bruckner and B. Thomson, ``Real analysis," Prentice-Hall, 1996.



%\bibitem{CZ}
%Y. Chen and K. Zhang, {\em Young measure solutions of the two-dimensional Perona-Malik equation in image processing},  Commun. Pure Appl. Anal., {\bf 5} (3) (2006),  615--635.

%\bibitem{CK}
%M. Chipot and D. Kinderlehrer,  {\em Equilibrium configurations of crystals}, Arch. Rational Mech. Anal., {\bf 103} (3) (1988), 237--277.


%\bibitem{CFG}
%D. Cordoba, D. Faraco and  F. Gancedo, {\em Lack of uniqueness for weak solutions of the incompressible porous media equation},  Arch. Ration. Mech. Anal., {\bf 200} (3) (2011), 725--746.


\bibitem{Da}
B. Dacorogna, ``Direct methods in the calculus of variations," Second edition. Applied Mathematical Sciences, 78. Springer, New York, 2008.

%\bibitem{DM1}
%B. Dacorogna and P. Marcellini, ``Implicit partial differential equations," Progress in Nonlinear Differential Equations and their Applications, 37. Birkh\"auser Boston, Inc., Boston, MA, 1999.

\bibitem{Dy}
W. Day, ``The thermodynamics of simple materials with fading memory," Tracts in Natural Philosophy, 22. Springer-Verlag, New York, Heidelberg and Berlin, 1970.

%\bibitem{DS}
%C. De Lellis and L. Sz\'ekelyhidi Jr, {\em  The Euler equations as a differential inclusion,} Ann. of Math., {\bf 170} (3) (2009), 1417--1436.

%\bibitem{DS2} C.~De Lellis and L.~Sz\'ekelyhidi Jr, {\em On admissibility criteria for weak solutions of the Euler equations,} Arch. %Rational Mech. Anal., {\bf 195} (2010) 225--260.

%\bibitem{Es}
%S. Esedoglu, {\em An analysis of the {P}erona-{M}alik scheme},  Comm. Pure Appl. Math., {\bf 54} (12) (2001),  1442--1487.

%\bibitem{EG} 	
%S. Esedoglu and J.B. Greer, {\em Upper bounds on the coarsening rate of discrete, ill-posed nonlinear diffusion equations},  Comm. Pure Appl. Math., {\bf  62} (1) (2009), 57--81.

%\bibitem{GG1}
%M. Ghisi and M. Gobbino, {\em A class of local classical solutions for the one-dimensional Perona-Malik equation}, Trans. Amer. Math. Soc., {\bf 361} (12) (2009),  6429--6446.

%\bibitem{GG2}
%M. Ghisi and M. Gobbino, {\em  An example of global classical solution for the Perona-Malik equation}, {Comm. Partial Differential Equations}, {\bf 36} (8) (2011),  1318--1352.

%\bibitem{Gr}
%M. Gromov, {\em Convex integration of differential relations},  Izv. Akad. Nauk SSSR Ser. Mat., {\bf 37} (1973), 329--343.

%\bibitem{Gu}
%P. Guidotti, {\em A backward-forward regularization of the Perona-Malik equation},  J. Differential Equations, {\bf 252} (4) (2012), 3226--3244.

\bibitem{Ho}
K. H\"ollig, {\em Existence of infinitely many solutions for a forward backward heat equation},  Trans. Amer. Math. Soc., {\bf 278} (1) (1983), 299--316.

%\bibitem{HN}
%K. H\"ollig and J.N. Nohel, {\em A diffusion equation with a nonmonotone constitutive function},   Systems of nonlinear partial differential equations (Oxford, 1982),  409�-422,
%NATO Adv. Sci. Inst. Ser. C: Math. Phys. Sci., 111, Reidel, Dordrecht-Boston, Mass., 1983.

%\bibitem{Ka}
%C. Kahane, {\em A gradient estimate for solutions of the heat equation. II}, Czechoslovak Math. J., {\bf 51} (126) (2001), 39--44.

%\bibitem{KK}
%B. Kawohl and N. Kutev, {\em Maximum and comparison principle for one-dimensional anisotropic diffusion}, Math. Ann., {\bf 311} (1) (1998),  107--123.

%\bibitem{Ky}
%S. Kichenassamy, {\em The Perona-Malik paradox},  SIAM J. Appl. Math., {\bf 57} (5) (1997),  1328--1342.

%\bibitem{KY}
%S. Kim and B. Yan, {\em Radial weak solutions for the Perona-Malik equation  as a differential inclusion}, J. Differential Equations, {\bf 258} (6) (2015), 1889--1932.

\bibitem{KY1}
S. Kim and B. Yan, {\em Convex integration and infinitely many weak solutions to the Perona-Malik equation in all dimensions}, SIAM J. Math. Anal., {\bf 47} (4) (2015), 2770--2794.

\bibitem{KY2}
S. Kim and B. Yan, {\em On Lipschitz solutions for some  forward-backward parabolic  equations}, Preprint.

%\bibitem{Ki}
%B. Kirchheim, {\em Rigidity and geometry of microstructures}, Habilitation thesis, University of Leipzig, 2003.

\bibitem{LSU}
O.A. Lady\v{z}enskaja and V.A. Solonnikov and N.N. Ural'ceva, ``Linear and quasilinear equations of parabolic type. (Russian)," Translated from the Russian by S. Smith. Translations of Mathematical Monographs, Vol. 23 American Mathematical Society, Providence, R.I. 1968.

\bibitem{Ln}
G.M. Lieberman,  ``Second order parabolic differential equations,"  World Scientific Publishing Co., Inc., River Edge, NJ, 1996.

%\bibitem{MP}
%S. M\"uller and M. Palombaro, {\em On a differential inclusion related to the Born-Infeld equations}, SIAM J. Math. Anal., {\bf 46} (4) (2014), 2385--2403.

%\bibitem{MSv1}
%S. M\"uller and V. \v Sver\'ak, {\em Convex integration with constraints and applications to phase transitions and partial differential equations}, J. Eur. Math. Soc. (JEMS), {\bf 1} (4) (1999), 393--422.

%\bibitem{MSv2}
%S. M\"uller and V. \v Sver\'ak, {\em Convex integration for Lipschitz mappings and counterexamples to regularity}, Ann. of Math. (2), {\bf 157} (3) (2003), 715--742.

%\bibitem{MSy}
%S. M\"uller and M. Sychev, {\em Optimal existence theorems for nonhomogeneous differential inclusions}, J. Funct. Anal., {\bf 181} (2) (2001), 447--475.

%\bibitem{Na}
%G. Nardi, {\em Schauder estimate for solutions of Poisson's equation with Neumann boundary condition,} arXiv:1302.4103
%(2013).

\bibitem{PM}
P. Perona and J. Malik, {\em Scale space and edge detection using anisotropic diffusion}, IEEE Trans. Pattern Anal. Mach. Intell., {\bf 12} (1990),  629--639.

%\bibitem{PW}
%M. Protter and H. Weinberger, ``Maximum principles in differential equations," Prentice-Hall, Inc., Englewood Cliffs, N.J. 1967.

%\bibitem{Po}
%L. Poggiolini, {\em Implicit pdes with a linear constraint}, Ricerche Mat.,  {\bf 52} (2)  (2003),   217--230.

%\bibitem{Sh}
%R. Shvydkoy, {\em Convex integration for a class of active scalar equations}, J. Amer. Math. Soc., {\bf 24} (4) (2011), 1159--1174.

%\bibitem{TTZ}
%S. Taheri, Q. Tang and Z. Zhang, {\em Young measure solutions and instability of the one-dimensional Perona-Malik equation}, J. Math. Anal. Appl., {\bf 308} (2) (2005), 467--490.

\bibitem{Tr}
C. Truesdell, ``Rational thermodynamics," 2nd ed., Springer-Verlag, New York, 1984.

%\bibitem{Ya1}
%B. Yan, {\em On the equilibrium set of magnetostatic energy by differential inclusion}, Calc. Var. Partial Differential Equations  {\bf 47}  (3-4) (2013),  547--565.

%\bibitem{Ya2}
%B. Yan, {\em On stability and asymptotic behaviours for a  degenerate Landau-Lifshitz equation}, Proc. Roy. Soc. Edinburgh Sect. A, to appear.

\bibitem{Zh}
K. Zhang, {\em Existence of infinitely many solutions for the one-dimensional Perona-Malik model}, Calc. Var. Partial Differential Equations, {\bf 26} (2) (2006), 171--199.

%\bibitem{Zh1}
%K. Zhang, {\em On existence of weak solutions for one-dimensional forward-backward diffusion equations}, J. Differential Equations, {\bf 220} (2) (2006), 322--353.

\end{thebibliography}
\end{document}